\documentclass[12pt,reqno]{amsart}

\usepackage{fixltx2e} 
\usepackage[utf8]{inputenc} 
\usepackage{amssymb, latexsym, stmaryrd, amsthm, dsfont, amsfonts, amsbsy, amsmath, mathrsfs} 
\usepackage{mathtools} 
\usepackage{bm, bbm} 
\usepackage{enumerate} 
\usepackage{verbatim} 
\usepackage{url}   
\usepackage{lscape} 
\usepackage{centernot}
\usepackage[dvipsnames]{xcolor}
\usepackage[colorinlistoftodos]{todonotes} 

\usepackage{microtype,tikz,tikz-cd}  
\makeatletter 
\def\MT@register@subst@font{\MT@exp@one@n\MT@in@clist\font@name\MT@font@list
 \ifMT@inlist@\else\xdef\MT@font@list{\MT@font@list\font@name,}\fi}
\makeatother

\makeatletter 
\newcommand{\myitem}[1]{%
\item[(#1)]\protected@edef\@currentlabel{#1}%
}
\makeatother

\usepackage[pdftex,bookmarks,bookmarksnumbered,linktocpage,   
         colorlinks,linkcolor=blue,citecolor=blue]{hyperref}

\newcommand{\bit}{\begin{itemize}}    
\newcommand{\eit}{\end{itemize}}
\newcommand{\ben}{\begin{enumerate}}
\newcommand{\een}{\end{enumerate}}
\newcommand{\benormal}{\ben[\normalfont 1.]}   

\let\enormal\een
\newcommand{\benroman}{\ben[\normalfont (i)]}  
\let\eroman\een

\newcommand{\bde}{\begin{description}}
\newcommand{\ede}{\end{description}}
\let\oper=\mathbb                               

\newcommand{\III}{\oper{I}}                     
\newcommand{\SSS}{\oper{S}}                     
\newcommand{\UUU}{\oper{U}}                     

\theoremstyle{theorem}
\newtheorem{Theorem}{Theorem}[section]
\newtheorem{Theorem-n}{Theorem}

\newtheorem{Proposition}[Theorem]{Proposition}
\newtheorem{Modal Sahlqvist Theorem}[Theorem]{Modal Sahlqvist Theorem}
\newtheorem{Intuitionistic Sahlqvist Theorem}[Theorem]{Intuitionistic  Sahlqvist Theorem}
\newtheorem{Esakia Duality}[Theorem]{Esakia Duality}

\newtheorem{Main Lemma}[Theorem]{Main Lemma}

\newtheorem{Transfer Lemma}[Theorem]{Transfer Lemma}
\newtheorem{Abstract Sahlqvist Theorem}[Theorem]{Abstract Sahlqvist Theorem}

\newtheorem{Lemma}[Theorem]{Lemma}
\newtheorem{Corollary}[Theorem]{Corollary}
\newtheorem{Claim}[Theorem]{Claim}

\theoremstyle{definition}
\newtheorem{Definition}[Theorem]{Definition}

\theoremstyle{remark}

\newtheorem{Remark}[Theorem]{Remark}

\let\leq=\leqslant
\let\nleq=\nleqslant
\let\geq=\geqslant 

\usepackage[mathscr]{euscript}
 \let\mathscr\relax 
\usepackage[scr]{rsfso}

\renewcommand{\int}{\mathsf{int}\,}

\bmdefine{\A}{A} 
\bmdefine{\C}{C}                                
                              
\bmdefine{\2}{2}
\bmdefine{\B}{B}
\bmdefine{\D}{D}

\usepackage{scalerel}

\newcommand{\pDL}{\mathsf{PDL}}
\newcommand{\DL}{\mathsf{DL}}
\newcommand{\pSP}{\mathsf{PSP}}

\newcommand{\up}{{\uparrow}}
\newcommand{\down}{{\downarrow}}
\newcommand{\V}{\mathsf{V}}

\newcommand{\ClopUp}{\mathsf{ClopUp}}
\newcommand{\Pries}{\mathsf{Pries}}

\renewcommand{\P}{\mathsf{P}}

\newcommand{\HHH}{\mathbb{H}}
\newcommand{\PPP}{\mathbb{P}}

\newcommand{\PPU}{\mathbb{P}_{\!\textsc{\textup{u}}}^{}}

\newcommand{\T}{\mathcal{T}}
\newcommand{\plt}{p_L^\T}
\newcommand{\fowedge}{\sqcap}
\newcommand{\fovee}{ \sqcup }
\newcommand{\fobigwedge}{\bigsqcap}
\newcommand{\fobigvee}{\bigsqcup}
\newcommand{\Jirr}{\mathsf{Jirr}}
\newcommand{\At}{\mathsf{At}}

\DeclareUnicodeCharacter{001B}{}


\begin{document}

\title[The universal theory of the free pseudocomplemented distributive lattice]{On the universal theory of the free pseudocomplemented distributive lattice}

\author{Luca Carai and Tommaso Moraschini}

\address{Luca Carai: Dipartimento di Matematica ``Federigo Enriques'', Universit\`a degli Studi di Milano, via Cesare Saldini 50, 20133 Milan, Italy}\email{luca.carai.uni@gmail.com}

\address{Tommaso Moraschini: Departament de Filosofia, Facultat de Filosofia, Universitat de Barcelona (UB), Carrer Montalegre, $6$, $08001$ Barcelona, Spain}
\email{tommaso.moraschini@ub.edu}

\date{\today}

\maketitle

\begin{abstract}
It is shown that the universal theory of the free pseudocomplemented distributive lattice is decidable and a recursive axiomatization is presented. This contrasts with the case of the full elementary theory of the finitely generated free algebras which is known to be undecidable. As a by-product, a description of the finitely generated pseudocomplemented distributive lattices that can be embedded into the free algebra is also obtained.
\end{abstract}

\section{Introduction}

A class of similar algebras is said to be a variety when it is closed under the formation of homomorphic images, subalgebras, and direct products. One of the main distinguishing features of varieties is that they contain free algebras with arbitrarily large sets of free generators (see, e.g., \cite[Thm.~II.10.12]{BS12}).
 This observation forms the main ingredient of the proof of a celebrated theorem of Birkhoff, stating that varieties coincide with equational classes (see, e.g., \cite[Thm.~II.11.9]{BS12}).
  In addition, each variety is generated by its countably-generated free algebra and, consequently, an equation holds in a variety if and only if it holds in its countably-generated free algebra.

Because of this, the problem of determining whether the elementary theory of free algebras of a given variety is decidable has attracted significant attention. For instance, a major problem raised by Tarski in 1945 asked whether any two free groups with two or more free generators are elementarily equivalent and whether the theory of free groups is decidable. Both problems were shown to have an affirmative answer in \cite{KM06,Sel06}. 
For the present purpose, other significant examples derive from the study of free Heyting algebras. On the one hand, the full elementary theory of the countably-generated free Heyting algebra is undecidable \cite{Rybakov1985}. On the other hand, its universal theory is decidable \cite{Ryb89} and, although not finitely axiomatizable, was given an independent infinite axiomatization in \cite{Jer08}.

In this paper, we focus on free pseudocomplemented distributive lattices.\ We recall that a \emph{pseudocomplemented distributive lattice} is an algebra $\langle A; \land, \lor, \lnot, 0, 1\rangle$ comprising a bounded distributive lattice $\langle A; \land, \lor, 0, 1 \rangle$ and a unary operation $\lnot$ that, when applied to an element $a \in A$, produces the largest element whose meet with $a$ is $0$ (see, e.g., \cite[Sec.~VIII]{BD74}). From a logical standpoint, the importance of pseudocomplemented distributive lattices derives from the fact that these form the implication-free subreducts of Heyting algebras (see, e.g.,  \cite[Proof of Thm.~2.6]{BP89}). As such, pseudocomplemented distributive lattices can be viewed as the algebraic counterpart of the implication-less fragment of the intuitionistic propositional calculus $\mathsf{IPC}$ in the same way  that Boolean algebras are related to classical logic (see, e.g., \cite{RV93}).

We recall that the class of pseudocomplemented distributive lattices is a variety and, therefore, free pseudocomplemented distributive lattices exist. The elementary theory of all finitely generated free pseudocomplemented distributive lattices was shown to be undecidable in  \cite{Idz87}.\footnote{Further decision problems related to pseudocomplemented distributive lattices were investigated in \cite{KS24b}.} In contrast, we show that the universal theory of all finitely generated free pseudocomplemented distributive lattices is decidable and exhibit a recursive axiomatization for it (Theorems \ref{Thm : decidable} and \ref{Thm : improved axiomatization}). Notably, this theory coincides with the universal theory of any $\kappa$-generated free pseudocomplemented distributive lattice with $\kappa$ infinite. Consequently, as the universal theory of the countably-generated free algebra of a variety coincides with the set of the so-called admissible universal sentences for that variety \cite[Thm.~2]{CM15} (see also \cite{Ryb97}),
our results can also be phrased in terms of admissibility  (Remark \ref{Rem : the final remark}).

Our main tools are Priestley duality for pseudocomplemented distributive lattices \cite{Pri75} and the description of free pseudocomplemented distributive lattices in terms of their Priestley duals in \cite{Urq73,DG80}. We also rely on the characterization of the  join-irreducibles of the finitely generated free pseudocomplemented distributive lattices in \cite{KS24}. Our strategy, on the other hand, hinges on two very general observations (Theorems \ref{Thm : class of models of universal theory : general} and \ref{Thm : recursively axiomatizable implies decidable}), which hold for each variety $\mathsf{V}$ that is finitely axiomatizable,  locally finite, and of finite type. Let $\mathsf{Th}_\forall(\boldsymbol{F})$ be the universal theory of the countably-generated free algebra $\boldsymbol{F}$ of $\mathsf{V}$. Then the following conditions hold:
\benroman
\item\label{item : 1 : introduction} the models of $\mathsf{Th}_\forall(\boldsymbol{F})$ are the members of $\mathsf{V}$ whose finite subalgebras embed into $\boldsymbol{F}$;
\item\label{item : 2 : introduction} if $\mathsf{Th}_\forall(\boldsymbol{F})$ is recursively axiomatizable, then it is also decidable.
\eroman
As the variety of pseudocomplemented distributive lattices is finitely axiomatizable, locally finite, and of finite type (see, e.g., \cite[Thm.~VIII.3.1]{BD74} and \cite[Thm.~4.55]{Ber11}), the observations above apply to the universal theory of the countably-generated free pseudocomplemented distributive lattice (from now on, simply the ``free algebra''). 

First, we introduce the notion of a poset with a \emph{free skeleton} (Section \ref{Sec : 5}) and show that a finite pseudocomplemented distributive lattice can be embedded into the free algebra if and only if its Priestley dual has a free skeleton (Proposition \ref{Prop : the main proposition}). When coupled with (\ref{item : 1 : introduction}), this yields that the models of the universal theory of the free algebra are precisely the pseudocomplemented distributive lattices whose finite subalgebras have a Priestley dual with a free skeleton (Theorem \ref{Thm : MAIN}). As the latter demand can be rendered as a recursive set of formulas (Proposition \ref{Prop : FSh, DN and free skeleton}), from (\ref{item : 2 : introduction}) it follows that the universal theory of the free pseudocomplemented distributive lattice is decidable (and axiomatized by this set).

\section{Universal theories, free algebras, and decidability}

Let $\III, \HHH, \SSS, \PPP$, and $\PPU$ be the class operators of closure under isomorphic copies, homomorphic images, subalgebras, direct products, and ultraproducts.\ A sentence is said to be \emph{universal} when it is of the form $\forall x_1, \dots, x_n P$, where $P$ is a quantifier-free formula. A class of similar algebras is \emph{universal} when it can be axiomatized by a set of universal sentences or, equivalently, when it is closed under $\III, \SSS$, and $\PPU$ (see, e.g., \cite[Thm.~V.2.20]{BS12}).\ As a consequence, the validity of universal sentences persists under $\III, \SSS$, and $\PPU$. Given a class of similar algebras $\mathsf{K}$, the least universal class containing $\mathsf{K}$ will be denoted by $\UUU(\mathsf{K})$.\ Notably, $\UUU(\mathsf{K}) = \III\SSS\PPU(\mathsf{K})$ (see, e.g., \cite[Thm.~V.2.20]{BS12}).

Let $\mathsf{K}$ be a class of similar algebras. Given a sentence $P$, we write $\mathsf{K}\vDash P$ to indicate that $P$ is valid in all the members of $\mathsf{K}$. When $\mathsf{K} = \{ \A \}$, we will  write $\A \vDash P$ as a shorthand for $\mathsf{K} \vDash P$. An algebra $\A$ is said to be a \emph{model} of a set of sentences $\Sigma$ when $\A \vDash P$ for every $P \in \Sigma$.  The set of universal sentences valid in $\mathsf{K}$ is called the \emph{universal theory} of $\mathsf{K}$ and will be denoted by $\mathsf{Th}_\forall(\mathsf{K})$. When $\mathsf{K} = \{ \A \}$, we will write  $\mathsf{Th}_\forall(\A)$ as a shorthand for $\mathsf{Th}_\forall(\mathsf{K})$. Clearly, the class of models of $\mathsf{Th}_\forall(\mathsf{K})$ coincides with $\UUU(\mathsf{K})$.

Given an algebra $\A$ and a set $X \subseteq A$, we denote by $\mathsf{Sg}^\A(X)$ the subuniverse of $\A$ generated by $X$. A subalgebra $\B$ of $\A$ is said to be \emph{finitely generated} when there exists a finite $X \subseteq A$ such that $B = \mathsf{Sg}^\A(X)$.
 We will make use of the following observation.

\begin{Proposition}\label{Prop : embedding into Pu}
Let $\mathsf{K} \cup \{ \A \}$ be a class of similar algebras.\ The following conditions hold:
\benroman
\item\label{item : 1 : Prop : embedding into Pu}  if every finitely generated subalgebra of $\A$ belongs to $\III\SSS(\mathsf{K})$, then $\A \in \UUU(\mathsf{K})$;
\item\label{item : 2 : Prop : embedding into Pu}  if $\A$ is finite and $\A \in \UUU(\mathsf{K})$, then $\A \in \III\SSS(\mathsf{K})$.
\eroman
\end{Proposition}

\begin{proof}
Recall that $\UUU(\mathsf{K}) = \III\SSS\PPU(\mathsf{K})$. For (\ref{item : 1 : Prop : embedding into Pu}), see \cite[Thm.\ V.2.14]{BS12}. For (\ref{item : 2 : Prop : embedding into Pu}), consider a finite $\A \in \UUU(\mathsf{K}) = \III\SSS\PPU(\mathsf{K})$. By \L o\'s' Theorem \cite[Thm.\ V.2.9]{BS12} we obtain $\A \in \III\SSS(\mathsf{K})$ as desired.
\end{proof}

A class of similar algebras is said to be a \emph{variety} when it can be axiomatized by a set of equations or, equivalently, when it is closed under $\HHH, \SSS$, and $\PPP$ (see, e.g., \cite[Thm.~II.11.9]{BS12}). While every variety is a universal class, the converse is not true in general. A class of algebras $\mathsf{K}$ is said to be \emph{locally finite} when $\mathsf{Sg}^\A(X)$ is finite for every $\A \in \mathsf{K}$ and finite $X \subseteq A$. We will rely on the next observation.

\begin{Theorem}\label{Thm : class of models of universal theory : general}
Let $\mathsf{V}$ be a locally finite variety and $\mathsf{K} \subseteq \mathsf{V}$. Then the class of models of $\mathsf{Th}_\forall(\mathsf{K})$ is
\[
\{ \A \in \mathsf{V} : \B \in \III\SSS(\mathsf{K}) \text{ for every finite subalgebra $\B$ of $\A$}\}.
\]
\end{Theorem}

\begin{proof}
Recall that the class of models of $\mathsf{Th}_\forall(\mathsf{K})$ is $\UUU(\mathsf{K})$. Therefore, it suffices to show that
\[
\UUU(\mathsf{K}) = \{ \A \in \mathsf{V} : \B \in \III\SSS(\mathsf{K}) \text{ for every finite subalgebra $\B$ of $\A$}\}.
\]

To prove the inclusion from left to right, consider $\A \in \UUU(\mathsf{K})$ and let $\B$ be a finite subalgebra of $\A$. Then $\B \in \SSS\UUU(\mathsf{K})$. Since $\UUU(\mathsf{K})$ is a universal class, it is closed under $\SSS$ and, therefore, $\B \in \UUU(\mathsf{K})$. As $\B$ is finite, we can apply Proposition \ref{Prop : embedding into Pu}(\ref{item : 2 : Prop : embedding into Pu}), obtaining $\B \in \III\SSS(\mathsf{K})$. Lastly, recall that $\mathsf{V}$ is a variety containing $\mathsf{K}$ by assumption. Therefore, $\UUU(\mathsf{K}) \subseteq \mathsf{V}$. Together with $\A \in \UUU(\mathsf{K})$, this yields $\A \in \mathsf{V}$ as desired.

Next we prove the inclusion from right to left. Let $\A \in \mathsf{V}$ be such that $\B \in \III\SSS(\mathsf{K})$ for every finite subalgebra $\B$ of $\A$. Since $\mathsf{K} \subseteq \mathsf{V}$ and $\mathsf{V}$ is locally finite by assumption, this implies that $\B \in \III\SSS(\mathsf{K})$ for every finitely generated subalgebra $\B$ of $\A$. By Proposition \ref{Prop : embedding into Pu}(\ref{item : 1 : Prop : embedding into Pu}) this yields $\A \in \UUU(\mathsf{K})$.
\end{proof}

A notable feature of varieties is that they contain \emph{free algebras} with arbitrarily large sets of free generators, as we proceed to recall.  Let $\mathsf{V}$ be a variety and $\kappa$ a positive cardinal. We denote the term algebra in the language of $\mathsf{V}$ with set of variables
$\{ x_\alpha : \alpha < \kappa \}$ by $\boldsymbol{T}_\mathsf{V}(\kappa)$. Then the \emph{free algebra} $\boldsymbol{F}_\mathsf{V}(\kappa)$ of $\mathsf{V}$ with $\kappa$ free generators is the quotient of $\boldsymbol{T}_\mathsf{V}(\kappa)$ under the congruence
\[
\theta_\kappa \coloneqq \{ \langle t, s\rangle \in T_\mathsf{V}(\kappa) \times T_\mathsf{V}(\kappa) : \mathsf{V} \vDash t \thickapprox s\}.
\]
As we mentioned, the free algebra $\boldsymbol{F}_\mathsf{V}(\kappa)$ always belongs to $\mathsf{V}$ (see, e.g., \cite[Thm.~II.10.12]{BS12}). Moreover, the set of free generators of $\boldsymbol{F}_\mathsf{V}(\kappa)$ is $\{ x_\alpha / \theta_\kappa : \alpha < \kappa \}$.  The next results collect some well-known properties of free algebras.

\begin{Proposition}\label{Prop : free algebra : technical prop}
 The following conditions hold for every variety $\mathsf{V}$ and positive cardinal $\kappa$: 
\benroman
\item\label{item : 1 : free algebra : technical prop} the free algebra $\boldsymbol{F}_\mathsf{V}(\kappa)$ is the direct limit  of the direct system whose objects are the subalgebras $\A_X$ of $\boldsymbol{F}_\mathsf{V}(\kappa)$ generated by a finite nonempty set $X \subseteq  \{ x_\alpha / \theta_\kappa : \alpha < \kappa \}$ and whose arrows are the inclusion maps $i \colon \A_Y \to \A_Z$ for $Y \subseteq Z$;
\item\label{item : 2 : free algebra : technical prop} for every $X \subseteq \{ x_\alpha / \theta_\kappa : \alpha < \kappa \}$ the subalgebra of $\boldsymbol{F}_\mathsf{V}(\kappa)$ generated by $X$ is isomorphic to $\boldsymbol{F}_\mathsf{V}(\lvert X \rvert)$.
\eroman
\end{Proposition}

\begin{proof}
Condition (\ref{item : 1 : free algebra : technical prop}) holds because $\boldsymbol{F}_\mathsf{V}(\kappa)$ is generated by $\{x_\alpha /\theta : \alpha < \kappa \}$, while  (\ref{item : 2 : free algebra : technical prop}) holds by construction.
\end{proof}

As a consequence of  (\ref{item : 2 : free algebra : technical prop}), each algebra $\A_X$ in condition (\ref{item : 1 : free algebra : technical prop}) is isomorphic to $\boldsymbol{F}_\mathsf{V}(n)$ for some $n \in \mathbb{Z}^+$. Therefore, from  (\ref{item : 1 : free algebra : technical prop}) we deduce the following.

\begin{Corollary}\label{Cor : fin gen free algebras cover everything}
Let $\mathsf{V}$ be a variety and $\kappa$ a positive cardinal. Then for each finitely generated subalgebra $\A$ of $\boldsymbol{F}_\mathsf{V}(\kappa)$ there exists $n \in \mathbb{Z}^+$ such that $\A$ embeds into $\boldsymbol{F}_\mathsf{V}(n)$.
\end{Corollary}

Although the following observation is folklore, we sketch a proof for the sake of completeness.

\begin{Theorem}\label{Thm : universal theory - free algebras}
Let $\mathsf{V}$ be a variety. Then for every infinite cardinal $\kappa$ we have
\[
\mathsf{Th}_\forall(\boldsymbol{F}_\mathsf{V}(\aleph_0)) = \mathsf{Th}_\forall(\boldsymbol{F}_\mathsf{V}(\kappa)) = \mathsf{Th}_\forall( \{ \boldsymbol{F}_\mathsf{V}(n) : n \in \mathbb{Z}^+ \}).
\]
\end{Theorem}

\begin{proof}
From Proposition \ref{Prop : free algebra : technical prop}(\ref{item : 2 : free algebra : technical prop}) it follows that $\boldsymbol{F}_\mathsf{V}(\aleph_0)$ embeds into $\boldsymbol{F}_\mathsf{V}(\kappa)$ and that each $\boldsymbol{F}_\mathsf{V}(n)$ embeds into $\boldsymbol{F}_\mathsf{V}(\aleph_0)$. Since the validity of universal sentences persists under the formation of subalgebras and isomorphic copies, this implies
\[
\mathsf{Th}_\forall(\boldsymbol{F}_\mathsf{V}(\kappa)) \subseteq \mathsf{Th}_\forall(\boldsymbol{F}_\mathsf{V}(\aleph_0)) \subseteq \mathsf{Th}_\forall( \{ \boldsymbol{F}_\mathsf{V}(n) : n \in \mathbb{Z}^+ \}).
\]
Therefore, it only remains to show that 
\[
\mathsf{Th}_\forall( \{ \boldsymbol{F}_\mathsf{V}(n) : n \in \mathbb{Z}^+ \}) \subseteq \mathsf{Th}_\forall(\boldsymbol{F}_\mathsf{V}(\kappa)).
\]
To this end, consider $P \in \mathsf{Th}_\forall( \{ \boldsymbol{F}_\mathsf{V}(n) : n \in \mathbb{Z}^+ \})$. We need to prove that $\boldsymbol{F}_\mathsf{V}(\kappa) \vDash P$. Since every finitely generated subalgebra of $\boldsymbol{F}_\mathsf{V}(\kappa)$ embeds into some $\boldsymbol{F}_\mathsf{V}(n)$ by Corollary \ref{Cor : fin gen free algebras cover everything}, we can apply Proposition \ref{Prop : embedding into Pu}(\ref{item : 1 : Prop : embedding into Pu}), obtaining 
\[
\boldsymbol{F}_\mathsf{V}(\kappa) \in \UUU(\{ \boldsymbol{F}_\mathsf{V}(n) : n \in \mathbb{Z}^+ \}).
\]
As $\{ \boldsymbol{F}_\mathsf{V}(n) : n \in \mathbb{Z}^+ \} \vDash P$ by assumption and $\UUU(\{ \boldsymbol{F}_\mathsf{V}(n) : n \in \mathbb{Z}^+ \})$ is the class of models of the universal theory of $\{ \boldsymbol{F}_\mathsf{V}(n) : n \in \mathbb{Z}^+ \}$, this implies $\boldsymbol{F}_\mathsf{V}(\kappa) \vDash P$.
\end{proof}

\begin{Remark}
Although we will not need it, we remark that Theorem \ref{Thm : universal theory - free algebras} can be strengthened as follows. Let $\mathsf{V}$ be a variety. Then for every infinite cardinal $\kappa$ the free algebras $\boldsymbol{F}_\mathsf{V}(\aleph_0)$ and $\boldsymbol{F}_\mathsf{V}(\kappa)$ are \emph{elementarily equivalent}, that is, satisfy exactly the same sentences \cite[Thm.\ 3.3(ii)]{TarskiVaught}. This is a consequence of the so-called Tarski-Vaught Test (see, e.g., \cite[Prop.~3.1.2]{CK90}), which can be applied to show that the natural embedding of $\boldsymbol{F}_\mathsf{V}(\aleph_0)$ into $\boldsymbol{F}_\mathsf{V}(\kappa)$ is elementary. 
\qed
\end{Remark}

In what follows we will assume some familiarity with the basics of computability theory (see, e.g., \cite{HMU07}). As we were unable to find a proof of the next result in the literature, we decided to sketch one for the sake of completeness.

\begin{Theorem}\label{Thm : recursively axiomatizable implies decidable}
Let $\mathsf{V}$ be a finitely axiomatizable and locally finite variety of finite type.\ If $\mathsf{Th}_\forall(\boldsymbol{F}_{\mathsf{V}}(\aleph_0))$ is recursively axiomatizable, then it is also decidable.
\end{Theorem}

\begin{proof}
Since $\mathsf{Th}_\forall(\boldsymbol{F}_{\mathsf{V}}(\aleph_0))$ is recursively axiomatizable, it is also recursively enumerable (see, e.g., \cite[Thm.~35I]{End01}). Therefore, in order to prove that $\mathsf{Th}_\forall(\boldsymbol{F}_{\mathsf{V}}(\aleph_0))$ is decidable, it suffices to show that the following set is recursively enumerable
\[
\{ P : P\text{ is a universal sentence such that }\boldsymbol{F}_{\mathsf{V}}(\aleph_0) \nvDash P \}. 
\]
In view of Theorem \ref{Thm : universal theory - free algebras}, the above set coincides with
\[
X \coloneqq \{ P : P\text{ is a universal sentence such that }\boldsymbol{F}_{\mathsf{V}}(n) \nvDash P \text{ for some }n \in \mathbb{Z}^+\}. 
\]

We turn our attention to proving that $X$ is recursively enumerable. Recall that each $\boldsymbol{F}_{\mathsf{V}}(n)$ is finite because $\mathsf{V}$ is locally finite by assumption. Therefore, if we can construct mechanically the various $\boldsymbol{F}_\mathsf{V}(n)$, we are done because if a universal sentence $P$ fails in some $\boldsymbol{F}_\mathsf{V}(n)$, we will be able to determine that this is the case in a finite amount of time by constructing the finite algebra $\boldsymbol{F}_\mathsf{V}(1)$ and testing if $P$ fails in it, then do the same for the finite algebra $\boldsymbol{F}_\mathsf{V}(2)$ if this is not the case and so on, until we reach some $n \in \mathbb{Z}^+$ such that $\boldsymbol{F}_\mathsf{V}(n) \nvDash P$.

Consequently, it only remains to show that each $\boldsymbol{F}_{\mathsf{V}}(n)$ can be constructed mechanically. Since $\mathsf{V}$ is finitely axiomatizable, we can start enumerating all the equations valid in $\mathsf{V}$ (see, e.g., \cite[Thm.~35I]{End01}).
 As $\boldsymbol{F}_{\mathsf{V}}(n)$ is finite and $\mathsf{V}$ of finite type, in a finite amount of time we will obtain a finite set of terms $\{ t_1, \dots, t_m \} \subseteq T_\mathsf{V}(n)$ containing $x_1, \dots, x_n$ such that for each basic $k$-ary operation $f$ and $t_{i_1}, \dots, t_{i_k}$ there exists $t_{j}$ with
\[
\mathsf{V} \vDash f(t_{i_1}, \dots, t_{i_k}) \thickapprox t_j.
\]
Then we form the algebra $\A$ with universe $\{ t_1, \dots, t_m \}$ and whose basic $k$-ary operations $f$ are defined by the rules given by the above equations, that is, by stipulating that
\[
f(t_{i_1}, \dots, t_{i_k}) = t_j.
\]
As $\A$ is finite and of finite type and $\mathsf{V}$ is finitely axiomatizable, we can find  the least congruence $\theta$ of $\A$ such that $\A / \theta \in \mathsf{V}$ in a finite amount of time.\ Lastly, from the construction of the free algebra $\boldsymbol{F}_\mathsf{V}(n)$ it follows that $\A / \theta \cong \boldsymbol{F}_\mathsf{V}(n)$.
\end{proof}

\section{Pseudocomplemented distributive lattices}

A \emph{pseudocomplemented distributive lattice} is an algebra $\A = \langle A; \land, \lor, \neg, 0, 1 \rangle$ which comprises a bounded distributive lattice $\langle A; \land, \lor, 0, 1 \rangle$ and a unary operation $\neg$ such that for all $a, b \in A$ we have
\[
a \leq \neg b \iff a \wedge b = 0.
\]
This means that $\neg b$ is the largest $a \in A$ such that $a \wedge b = 0$. The element $\neg b$ is called the \emph{pseudocomplement} of $b$. 
Equivalently, pseudocomplemented distributive lattices are the implication-less subreducts of Heyting algebras ${\A = \langle A; \land, \lor, \lnot, \to, 0, 1 \rangle}$ (see, e.g., \cite[Proof of Thm.~2.6]{BP89}). 

We will make use of the following property of pseudocomplements (see, e.g., \cite[Lem.~3.33]{Ber11}).

\begin{Proposition}\label{Prop : triple negation}
Let $a$ be an element of a pseudocomplemented distributive lattice. Then $\neg \neg \neg a = \neg a$.
\end{Proposition}

It is well known that the class of pseudocomplemented distributive lattices forms a finitely axiomatizable variety (see, e.g., \cite[Thm.~VIII.3.1]{BD74}). Furthermore, the following holds (see, e.g., \cite[Thm.~4.55]{Ber11}).

\begin{Theorem}\label{Thm : PDL locally finite}
The variety of pseudocomplemented distributive lattices is locally finite.
\end{Theorem}

Notably, every finite distributive lattice $\langle A; \land, \lor, 0, 1 \rangle$ can be viewed as a pseudocomplemented one by setting 
\[
\lnot a \coloneqq \max \{ b \in A : a \land b = 0 \} \, \, \text{ for each }a\in A.
\]
The same is true for the lattice of open sets of every topological space by stipulating that if $U$ is an open set, then $\lnot U$ is the topological interior of $U^c$.

In view of 	\emph{Priestley duality} \cite{Pri70,Pri72}, bounded distributive lattices can be studied through the lenses of duality theory, as we proceed to recall. Given a poset $\langle X; \leq \rangle$ and $Y \subseteq X$, let
\begin{align*}
\up Y &\coloneqq \{x\in X : \text{there exists } y \in Y \text{ such that } y\leq x\};\\
\down Y &\coloneqq \{x\in X : \text{there exists } y \in Y \text{ such that } x\leq y\}.
\end{align*}
The set $Y$ is said to be an {\em upset} (resp.\ {\em downset}) if $Y={\uparrow}Y$ (resp.\ $Y={\downarrow}Y$). When $Y = \{ x \}$, we will write ${\uparrow} x$ and ${\downarrow} x$ instead of ${\uparrow} \{ x \}$ and ${\downarrow} \{ x\}$.
For all $x, y \in X$ with $x \leq y$, we let $[x,y] \coloneqq \{ z \in X : x \leq z \leq y\}$.

An \emph{ordered topological space} is a triple $\langle X, \leq, \tau \rangle$ comprising a poset $\langle X, \leq \rangle$ and a topology $\tau$ on $X$. Given an ordered topological space $X$, we denote the set of its clopen upsets by $\ClopUp(X)$. An ordered topological space $X$ is said to be a \emph{Priestley space} when it is compact and satisfies the \emph{Priestley separation axiom}, namely, the demand that for all $x, y \in X$,
\[
\text{if }x \nleq y \text{, there exists }U \in \ClopUp(X) \text{ such that }x \in U \text{ and }y \notin U.
\]
A map $p \colon X \to Y$ between Priestley spaces is said to be a \emph{Priestley morphism} when it is continuous and order preserving. We denote the category of Priestley spaces with Priestley morphisms between them by $\Pries$. Similarly, we denote the category of bounded distributive lattices with homomorphisms between them by $\DL$.

Priestley duality establishes a dual equivalence between $\DL$ and $\Pries$. More precisely, let $\A$ be a bounded distributive lattice. A set $F \subseteq A$ is a \emph{prime filter} of $\A$ when it is a nonempty proper upset such that for all $a, b \in A$,
\[
(a, b \in F \Longrightarrow a \land b \in F) \qquad \text{and} \qquad (a \lor b \in F \Longrightarrow a \in F \text{ or }b \in F).
\]
We denote the set of prime filters of $\A$ by $\mathsf{Pr}(\A)$. Now, for each $a \in A$ let
\[
\gamma_\A(a) \coloneqq \{ F \in \mathsf{Pr}(A) : a \in F \}.
\]
Then the triple $\A_* \coloneqq \langle \mathsf{Pr}(\A), \subseteq, \tau \rangle$, where  $\tau$ is the topology on $\mathsf{Pr}(\A)$ generated by the subbasis
\[
\{ \gamma_\A(a) : a \in A \} \cup \{ \gamma_\A(a)^c : a \in A \},
\]
is a Priestley space. Furthermore, given a homomorphism $h \colon \A \to \B$ between bounded distributive lattices, the map $h_\ast \colon \B_* \to \A_*$ defined by the rule $h_*(F) \coloneqq h^{-1}[F]$ is a Priestley morphism. Notably, the transformation $(-)_* \colon \DL \to \Pries$ can be viewed as a contravariant functor.

On the other hand, given a Priestley space $X$, the structure $X^* \coloneqq \langle \ClopUp(X); \cap, \cup, \varnothing, X \rangle$ is a bounded distributive lattice. Moreover, given a Priestley morphism $p \colon X \to Y$, the map $p^\ast \colon Y^* \to X^*$ defined by the rule $p^*(U) \coloneqq p^{-1}[U]$ is a homomorphism of bounded distributive lattices. Lastly, the transformation $(-)^* \colon \Pries \to \DL$ can also be viewed as a contravariant functor.

\begin{Theorem}[Priestley duality]
The functors $(-)_*$ and $(-)^*$ witness a dual equivalence between the categories $\DL$ and $\Pries$.
\end{Theorem}

As Priestley spaces are Hausdorff, the topology of each finite Priestley space is discrete. Therefore, the full subcategory of $\Pries$ consisting of finite Priestley spaces is isomorphic to the category of finite posets with order preserving maps between them. Together with the fact that Priestley duality preserves the property of being finite, we obtain the following (see, e.g., \cite[Thm.\ 1.25]{GvG24}).

\begin{Theorem}[Finite Priestley duality]\label{Thm : finite Priestley duality}
The category of finite bounded distributive lattices with homomorphisms between them is dually equivalent to that of finite posets with order preserving maps between them.
\end{Theorem}

The dual $\A_\ast$ of a finite distributive lattice $\A$ can be identified with the order dual of a subposet of $\A$, as we proceed to recall. Let $\At(\A)$ and $\Jirr(\A)$ be the sets of atoms and of join-irreducibles of $\A$, respectively. Recall that the join-irreducibles are assumed to be different from $0$. We view $\Jirr(\A)$ as a poset with the order induced by the one of $\A$.
Then consider the map $\up ( - ) \colon \Jirr(\A) \to \A_*$ sending $a$ to $\up a$ and the map $\min \colon \A_* \to \Jirr(\A)$ sending $F \in \A_*$ to its least element $\min F$. Moreover, given a poset $X$ and $Y \subseteq X$, we denote the set of maximal elements of the subposet of $X$ with universe $Y$ by $\max Y$. In view of the next observation (see, e.g., \cite[p.\ 52]{GvG24}), $\A_\ast$ can be identified with $\Jirr(\A)$ and $\max \A_\ast$ with $\At(\A)$.

\begin{Proposition}\label{Prop : correspondence prime filters join irr}
The maps $\up ( - )$ and $\min$ are well-defined dual isomorphisms inverse of one another. Moreover, they restrict to bijections between $\At(\A)$ and $\max \A_*$.
\end{Proposition}

Lastly, the following is an immediate consequence of the finiteness of $\A$.

\begin{Proposition}\label{Prop : atoms underneath}
Let $\A$ be a finite distributive lattice. For every $a \in A$ distinct from the minimum of $\A$ there exists $b \in \At(\A)$ such that $b \leq a$.
\end{Proposition}

Let $\pDL$ be the category of pseudocomplemented distributive lattices with homomorphisms between them. As $\pDL$ is a subcategory of $\DL$, it is clear that Priestley duality restricts to a duality between $\pDL$ and a suitable subcategory of $\Pries$.

\begin{Definition}
\
\benroman
\item A Priestley space  $X$ is said to be a \emph{p-space} when $\down U$ is clopen for every $U \in \ClopUp(X)$. 
\item A map $p \colon X \to Y$ between posets is said to be a \emph{weak p-morphism} when it  is  order preserving  and for all $x \in X$ and $y \in \max Y$,
\[
\text{if }p(x) \leq y\text{, there exists }z \in \max \up x\text{ such that }p(z)=y.
\]
\eroman
\end{Definition}

The category of p-spaces with continuous weak p-morphisms between them will be denoted by $\pSP$.

\begin{Theorem}\label{Thm : duality for PDL}\cite{Pri75}
Priestley duality 
 yields
a dual equivalence between $\pDL$ and $\pSP$.
\end{Theorem}

Theorems \ref{Thm : finite Priestley duality} and \ref{Thm : duality for PDL} yield the following finite duality.

\begin{Corollary}\label{Cor : finite PDL duality}
The category of finite pseudocomplemented distributive lattices  with homomorphisms between them is dually equivalent to that of finite posets with weak p-morphisms between them.
\end{Corollary}

As we observed above, every finite distributive lattice can be seen as a pseudocomplemented distributive lattice. However, a homomorphism of bounded lattices between finite pseudocomplemented distributive lattices does not necessarily preserve pseudocomplements. So, the main difference between the categories of finite distributive lattices and of finite pseudocomplemented distributive lattices lies in their morphisms. This is reflected in the fact that an order preserving map between finite posets is not necessarily a weak p-morphism.

We will rely on the following observation.

\begin{Proposition}\label{Prop : preservation of maximal}
Let $p \colon X \to Y$ be a weak p-morphism between finite posets. Then 
\benroman
\item\label{item : preservation of maximal : 1} $p[\max \up x]=\max \up p(x)$ for every $x \in X$;
\item\label{item : preservation of maximal : 2} $p[\max X] \subseteq \max Y$.
\eroman
\end{Proposition}

\begin{proof}
(\ref{item : preservation of maximal : 1}):
To prove the inclusion from left to right, consider $y \in \max \up x$. Since $Y$ is  finite, there exists $z \in \max Y$ such that $p(y) \leq z$. As $p$ is a weak p-morphism, there exists $u \in \max \up y$ such that $p(u) = z$. Together with the assumption that $y$ is a maximal element of $X$, this yields $y = u$ and, therefore, $p(y) = p(u) = z \in \max \up p(y)$. It follows that $p(y) \in \max \up p(x)$ because $x \leq y$ and $p$ is order preserving. Then we prove the inclusion from right to left. Let $y \in \max \up p(x)$. Then $p(x) \leq y$. Since $y \in \max Y$ and $p$ is a weak p-morphism, there exists $z \in \max \up x$ such that $p(z)=y$. Therefore, $y \in p[\max \up x]$.

(\ref{item : preservation of maximal : 2}): Let $x \in \max X$. From (\ref{item : preservation of maximal : 1}) it follows that $\{p(x)\} = p[\max \up x] = \max \up p(x) \subseteq \max Y$. Thus, $p(x) \in \max Y$.
\end{proof}

Under Priestley duality, embeddings in $\DL$ correspond to the morphisms in $\Pries$ that are surjective (see, e.g., \cite[Prop.~11]{Pri72}). As a consequence of Theorem~\ref{Thm : duality for PDL}, a similar dual characterization holds for embeddings in $\pDL$. For our purposes it will be sufficient to consider embeddings between finite pseudocomplemented distributive lattices.
We say that a poset $Y$ is a \emph{weak p-morphic image} of a poset $X$ when there exists a surjective weak p-morphism from $X$ onto $Y$.

\begin{Proposition}\label{Prop : weak p-morphic image}
Let $\A, \B \in \pDL$ be finite. Then $\A$ embeds into $\B$ if and only if $\A_*$ is a weak p-morphic image of $\B_*$.
\end{Proposition}

\section{Free extensions}

We rely on the following universal construction (see, e.g., \cite[Ch.~IV]{Mac71}).

\begin{Definition}
A pseudocomplemented distributive lattice $\B$ is \emph{free over a bounded distributive lattice} $\A$ via a bounded lattice homomorphism $e \colon \A \to \B$ when for every $\C \in \pDL$ and bounded lattice homomorphism $f \colon \A \to \C$ there exists a unique $\pDL$-morphism $g \colon \B \to \C$ such that $g \circ e = f$. 
\[
\begin{tikzcd}[sep = large]
\B \arrow[r, dashed, "\exists ! \,g"] & \C \\
\A \arrow[ur, "f"'] \arrow[u, "e"] & 
\end{tikzcd}
\]
\end{Definition}

Notably, the pseudocomplemented distributive lattice free over a distributive lattice always exists  and is unique up to isomorphism (see, e.g., \cite[Cor.~IV.1.1]{Mac71}). Furthermore, it admits a very transparent dual description \cite{DG80}, as we proceed to recall.

Let $X$ be a Priestley space.\ The dual of the pseudocomplemented distributive lattice free over $X^*$ can be described by means of the so-called \emph{Vietoris space} $\V(X)$ of $X$ \cite{Mic51}. As the order-free reduct of the Priestley space $X$ is a Stone space, the Vietoris space $\V(X)$ can be described as follows \cite[Lem.~2.7]{KKV04}: it is the set of all the nonempty closed subsets of $X$ equipped with the topology generated by the subbasis $\{ \Box V, \Diamond V : V \text{ is a clopen subset of } X\}$, where
\begin{align*}
\Box V \coloneqq \{ C \in \V(X) : C \subseteq V \} \qquad \text{and} \qquad \Diamond V \coloneqq \{ C \in \V(X) : C \cap V \ne \varnothing \}.
\end{align*}

\begin{Definition}
Given a Priestley space $X$, we let
\[
\P(X) \coloneqq \{ \langle x, C \rangle  \in X \times \V(X) : C \subseteq \up x \}.
\]
We endow $\P(X)$ with the subspace topology induced by $X \times \V(X)$ and we order it as follows:
\[
\langle x, C \rangle \leq \langle y, D \rangle \iff x \leq y \ \text{and} \ C \supseteq D.
\]
\end{Definition}

\begin{Theorem}\cite[Thm.~2.2]{DG80}\label{Thm : free extensions PDL}
Let $X$ be a Priestley space. Then $\P(X)$ is the dual of the pseudocomplemented distributive lattice free over $X^*$.
\end{Theorem}

\begin{Remark}\label{Rem : finite free extensions}
When the Priestley space $X$ is finite, the definition of $\P(X)$ admits the following simplification. Recall that, in this case, the topology of $X$ is discrete. Consequently, every subset of $X$ is closed and
\[
\P(X) = \{ \langle x, C \rangle  \in X \times \wp(X) : \varnothing \ne C \subseteq \up x \},
\]
 where $\wp(X)$ denotes the powerset of $X$.

Furthermore, in this case, $\P(X)$ is a finite Priestley space, whose topology is also discrete. Therefore, we can identify $\P(X)$ with its underlying finite poset, thus making it amenable to the finite duality of Corollary \ref{Cor : finite PDL duality}.
\qed
\end{Remark}

Free extensions are tightly connected to free algebras as follows. The free pseudocomplemented distributive lattice $\boldsymbol{F}_\pDL(\kappa)$ is isomorphic to the free extension of the free bounded distributive lattice $\boldsymbol{F}_\DL(\kappa)$. The dual of the latter is the Priestley space $\mathsf{2}^\kappa$, obtained by considering the direct power of the two-element chain $\mathsf{2}$ with universe $\{ 0, 1 \}$, viewed as Priestley space, endowed with the product topology (see, e.g., \cite[Prop.~4.8]{GvG24}). Consequently, from Theorem \ref{Thm : free extensions PDL} we deduce the following.

\begin{Corollary}\cite[Cor.~2.3(i)]{DG80}\label{Cor : free pDL power of two}
Let $\kappa$ be a cardinal. Then $\P(2^\kappa)$ is the dual of $\boldsymbol{F}_\pDL(\kappa)$.
\end{Corollary}

\noindent In the case where $\kappa$ is finite, the above result was first discovered in \cite{Urq73}.

\section{Posets with a free skeleton}\label{Sec : 5}

The following concept will play a fundamental role in this paper.

\begin{Definition}\label{def}
A poset $X$ with minimum $\bot$ is said to have a \emph{free skeleton} when the following hold:
\benroman
\item\label{item : def : 1} for all $x \in X$ and nonempty $Y \subseteq \max {\uparrow}x$ there exists an element $s_{x,Y} \in {\uparrow}x$ such that
\[
Y = \max {\uparrow}s_{x,Y};
\]
\item\label{item : def : 2}  for all $x \in X$ and nonempty $Y, Z \subseteq \max {\uparrow}x$,
\[
Y \subseteq Z \text{ implies }s_{x, Z} \leq s_{x, Y};
\]
\item\label{item : def : 3} for all $x \in X$ and nonempty $Y \subseteq \max X$,
\[
\max {\uparrow}x \subseteq Y \text{ implies }s_{\bot, Y} \leq x.
\]
\eroman
\end{Definition}

As shown by the next result, the structure of posets with a free skeleton is very permissive.

\begin{Proposition}\label{Prop : bounded posets have a free skeleton}
Every bounded poset has a free skeleton.
\end{Proposition}

\begin{proof}
Let $X$ be a poset with minimum $\bot$ and maximum $\top$.\ Clearly, for every $x \in X$ the only nonempty subset of $\max \up x$ is $Y \coloneqq \{ \top \}$. Then for each $x \in X$ let $s_{x, Y} \coloneqq x$.
Clearly, $s_{x, Y} \in \up x$. Moreover, $Y = \max \up s_{x, Y}$ because $\top$ is the maximum of $X$ and $Y = \{ \top \}$. Therefore, condition (\ref{item : def : 1}) of Definition \ref{def} holds. Condition (\ref{item : def : 2}) of the same definition holds because $Y$ is the only nonempty subset of $\max X$. Lastly, condition (\ref{item : def : 3}) holds because $s_{\bot, Y}$ is the minimum of $X$ by definition. Hence, the poset $X$ has a free skeleton.
\end{proof}

The next result isolates a method for constructing new posets with a free skeleton from given ones.

\begin{Proposition}\label{Prop : skeletal upsets}
Let $X$ be a poset with a free skeleton and minimum $\bot$. For each nonempty $Y \subseteq \max X$ the poset ${\uparrow}s_{\bot, Y}$ has also a free skeleton.
\end{Proposition}

\begin{proof}
We will define a family
\[
S \coloneqq \{ s^\ast_{x, Z} : x \in \up s_{\bot, Y} \text{ and }\varnothing \ne Z \subseteq \max \up x \}
\]
that globally satisfies the conditions of Definition \ref{def} for the poset ${\uparrow}s_{\bot, Y}$. For each $x \in \up s_{\bot, Y}$ and nonempty $Z \subseteq \max \up x$ let
\[
s^\ast_{x, Z} \coloneqq  \begin{cases}
		  s_{x, Z}  & \text{ if } x > s_{\bot, Y};\\
		s_{\bot, Z} & \text{ if }x = s_{\bot, Y}.
		\end{cases}
\]

Since $x \in \up s_{\bot, Y}$, condition (\ref{item : def : 1}) of Definition \ref{def} yields
\[
\max \up s_{x, Z} = \max \up s_{\bot, Z} = Z \subseteq \max \up x \subseteq \max {\uparrow}s_{\bot, Y} = Y.
\]
By condition (\ref{item : def : 3}) of Definition \ref{def} we conclude that $s_{\bot, Y} \leq s_{x, Z}, s_{\bot, Z}$.\ It follows that $s^\ast_{x, Z}$ belongs to ${\uparrow}s_{\bot, Y}$. 
Therefore, it only remains to prove that the family $S$ globally satisfies the conditions of Definition \ref{def} for the poset ${\uparrow}s_{\bot, Y}$. To this end, we will use without further notice the fact that the family $\{
s_{x, Z} : x \in X \text{ and }\varnothing \ne Z \subseteq \max \up x 	\}$ witnesses the property of ``having a free skeleton'' for $X$.

Clearly, $\up s_{\bot, Y}$ has a minimum. Moreover, condition (\ref{item : def : 1}) of Definition \ref{def} holds because for each $x \in \up s_{\bot, Y}$ and nonempty $Z \subseteq \max \up x$ we have $Z = \max \up s_{x, Z} = \max \up s_{\bot, Z}$, whence $Z = \max \up s^\ast_{x, Z}$ by the definition of $s^\ast_{x, Z}$. To prove condition (\ref{item : def : 2}), consider $x \in \up s_{\bot, Y}$ and nonempty $Z_1, Z_2 \subseteq \max {\uparrow}x$ such that $Z_1 \subseteq Z_2$. Notice that $Z_1, Z_2 \subseteq \max {\uparrow}x \cap \max {\uparrow}\bot$. Consequently, $s_{x, Z_2} \leq s_{x, Z_1}$ and $s_{\bot, Z_2} \leq s_{\bot, Z_1}$. Together with the definition of $s^\ast_{x, Z_1}$ and $s^*_{x, Z_2}$, this yields 
 $s^*_{x,Z_2} \leq s^*_{x,Z_1}$. Therefore, it only remains to prove condition (\ref{item : def : 3}). Observe that $s_{\bot, Y}$ is the minimum of $\up s_{\bot, Y}$ and consider $x \in \up s_{\bot, Y}$ and a nonempty $Z \subseteq \max \up s_{\bot, Y}$ such that $\max \up x \subseteq Z$. From $x \in X$, $\varnothing \ne Z \subseteq \max X$, and $\max \up x \subseteq Z$ it follows that $s_{\bot, Z} \leq x$. By the definition of $s^*_{s_{\bot, Y}, Z}$ this amounts to $s^*_{s_{\bot, Y}, Z} \leq x$. As $s_{\bot, Y}$ is the minimum of $\up s_{\bot, Y}$, this establishes condition (\ref{item : def : 3}) for the poset $\up s_{\bot, Y}$.
\end{proof}

Lastly, we will rely on the next observations.

\begin{Proposition}\label{Prop : extension}
Let $X$ and $Y$ be finite posets such that $Y$ has a free skeleton and let $U$ be an upset of $X$. Every weak p-morphism $p \colon U \to Y$ can be extended to a weak p-morphism $p^+ \colon X \to Y$.
\end{Proposition}

\begin{proof}
Since $Y$ has a free skeleton, it is nonempty. As it is also finite, $\max Y$ is nonempty. So, we can fix an element $w$ of $\max Y$. For each $x \in X$ let
\[
d(x) \coloneqq p[U \cap \max \up x]\qquad \text{and} \qquad
e(x) \coloneqq  \begin{cases}
		  d(x)  & \text{ if } \max {\uparrow}x \subseteq U;\\
		d(x) \cup \{ w \} & \text{ otherwise}.
		\end{cases}
\]
Notice that for each $x \in X$ the set $e(x)$ is a subset of $\max Y$ by Proposition~\ref{Prop : preservation of maximal}(\ref{item : preservation of maximal : 2}).
Since $X$ is finite, $\max \up x \subseteq U$ implies $U \cap \max \up x = \max \up x \ne \varnothing$. Thus, $e(x)$ is nonempty for every $x \in X$.

Now, let $p^+ \colon X \to Y$ be the map defined for every $x \in X$ as
\[
p^+(x) \coloneqq  \begin{cases}
		 p(x) & \text{ if } x \in U ;\\
		s_{\bot, e(x)} & \text{ if }x \notin U \cup \max X;\\
		w & \text{ if }x \in \max X - U.
		\end{cases}
\]
As the three cases in the definition of $p^+$ are mutually exclusive and each element of $X$ satisfies one of them, the map $p^+ \colon X \to Y$ is a well-defined map that extends $p \colon U \to Y$. It only remains to prove that $p^+$ is a weak p-morphism. 
We begin with the following observation:

\begin{Claim}\label{Claim : e is max}
For every $x \in X$ we have $e(x) = \max \up p^+(x)$.
\end{Claim}

\begin{proof}[Proof of the Claim]
We will consider the three cases in the definition of $p^+$ separately. First, suppose that $x \in U$. Since $U$ is an upset containing $x$, we have $\max {\uparrow}x \subseteq U$. By the definition of $e$ we also have $e(x) = d(x) = p[U \cap \max \up x] = p[\max \up x]$. As $p \colon U \to Y$ is a weak p-morphism, we obtain $p[\max \up x] = \max \up p(x)$ by Proposition~\ref{Prop : preservation of maximal}(\ref{item : preservation of maximal : 1}). Thus, $e(x) = \max \up p^+(x)$ as desired because $p(x)=p^+(x)$.

 In the case where $x \notin U \cup \max X$, we have $\max \up p^+(x) = \max \up s_{\bot, e(x)} = e(x)$ by the definition of $p^+$ and condition~(\ref{item : def : 1}) of Definition \ref{def}. Therefore, it only remains to consider the case where $x \in \max X - U$. We have $U \cap \max \up x = U \cap \{x\}= \varnothing$. The definitions of $e$ and $p^+$ yield $e(x)=d(x) \cup \{w\}=\{w\}$ and $\max \up p^+(x) = \max \up w = \{w\}$ because $w \in \max Y$.
\end{proof}

Now, we prove that $p^+$ is order preserving. To this end, let $x, y \in X$ be such that $x < y$. We may assume that either $x$ or $y$ does not belong to $U$, otherwise we are done because $p \colon U \to Y$ is order preserving and $p^+$ extends $p$. Since $U$ is an upset and $x < y$, this means that $x \notin U$. As $x < y$, we also have $x \notin \max X$. Then $x \notin U \cup \max X$. By the definition of $p^+$ we obtain $p^+(x) = s_{\bot, e(x)}$. In view of Claim \ref{Claim : e is max}, we have $e(y) = \max {\uparrow}p^+(y)$. 
From $x \leq y$ it follows that 
$d(y) \subseteq d(x)$. Since $\max \up x \subseteq U$ implies $\max \up y \subseteq U$, we also have $e(y) \subseteq e(x)$.
Therefore, $\max {\uparrow}p^+(y) \subseteq e(x)$. By condition (\ref{item : def : 3}) of Definition \ref{def}  we conclude that $p^+(x) = s_{\bot, e(x)} \leq p^+(y)$ as desired.

It only remains to prove that $p^+$ satisfies the weak p-morphism condition.\ Consider $x \in X$ and $z \in \max {\uparrow}p^+(x)$.\ From Claim \ref{Claim : e is max} we obtain $z \in e(x)$.\ We have two cases:\ either $z \in d(x)$ or $z \notin d(x)$. First suppose that $z \in d(x)$. Then there exists $y \in U \cap \max \up x$ such that $p(y) = z$ and we are done. Then we consider the case where $z \notin d(x)$. Since $z \in e(x)$, we obtain $z \in e(x) - d(x)$. By the definition of $e$ this implies that $z = w$ and there exists $y \in \max \up x - U$. So, $y \in \max X - U$, and the definition of $p^+$ yields $p^+(y) = w = z$. Since $y \in \max \up x$, this establishes that $p^+$ is a weak p-morphism.
\end{proof}

\begin{Proposition}\label{Prop : equiv free skel}
Let $\A \in \pDL$ be finite. Then $\A_*$ has a free skeleton if and only if $\A$ is nontrivial and the following  holds:
\benroman
\item\label{item : equiv free skel : 1} for all $a \in \Jirr(\A)$ and nonempty $Y \subseteq \At(\A) \cap \down a$ there exists $c_{a,Y} \in \Jirr(\A) \cap \down a$ such that 
\[
Y = \At(\A) \cap \down c_{a,Y};
\]
\item\label{item : equiv free skel : 2}  for all $a \in \Jirr(\A)$ and nonempty $Y, Z \subseteq \At(\A) \cap \down a$,
\[
Y \subseteq Z \text{ implies }c_{a, Y} \leq c_{a, Z};
\]
\item\label{item : equiv free skel : 3} for all $b \in A - \{ 0 \}$ we have $\neg \neg b \in \Jirr(\A)$.
\eroman
\end{Proposition}

\begin{proof}
We first establish the following claim.

\begin{Claim}\label{Claim : properties double neg and atoms}
Let $c,d \in A$. Then $\At(\A) \cap \down c = \At(\A) \cap \down \neg \neg c$ and
\[
\At(\A) \cap \down c \subseteq \At(\A) \cap \down d \text{ implies } c \leq \neg \neg d.
\]
\end{Claim}

\begin{proof}[Proof of the Claim]
We begin by showing that $\At(\A) \cap \down c = \At(\A) \cap \down \neg \neg c$. The inclusion from left to right holds because $c \leq \neg \neg c$. To prove the other inclusion, consider $e \in \At(\A) \cap \down \neg \neg c$. Then $e \leq \neg \neg c$, and so $e \wedge \neg c = 0$. Since $e$ is an atom, if $e \nleq c$, then $e \wedge c = 0$, and hence $e \leq \neg c$. But this is impossible because otherwise $e = e \wedge \neg c = 0$,  a contradiction with $e \in \At(\A)$. Thus, we conclude that $e \leq c$.\ This establishes the inclusion from right to left. Hence, $\At(\A) \cap \down c = \At(\A) \cap \down \neg \neg c$.

Now, assume that $\At(\A) \cap \down c \subseteq \At(\A) \cap \down d$. We prove that $c \leq \neg \neg d$. Suppose, with a view to contradiction, that $c \nleq \neg \neg d$. Then $c \wedge \neg d \ne 0$. Since $\A$ is finite, there exists $e \in \At(\A)$ with $e \leq c \wedge \neg d$ (see Proposition~\ref{Prop : atoms underneath}). Consequently $e \leq c$, but $e \nleq d$  otherwise $e \leq d \wedge \neg d = 0$, a contradiction with $e \in \At(\A)$. So, $e \in  \At(\A) \cap \down c$ and $e \notin \At(\A) \cap \down d$, contradicting the hypothesis. Thus, we conclude that $c \leq \neg \neg d$.
\end{proof}

Recall from Proposition~\ref{Prop : correspondence prime filters join irr} that the maps $\up(-) \colon \Jirr(\A) \to \A_*$ and $\min \colon \A_\ast \to \Jirr(\A)$ are dual order isomorphisms inverse of one another, which restrict to bijections between $\At(\A)$ and $\max \A_*$. Moreover, $\A_*$ has a least element if and only if $1 \in \Jirr(\A)$ because $\A$ is finite.
It is then straightforward to check that $\A_*$ has a free skeleton if and only if $1 \in \Jirr(\A)$ and $\A$ satisfies conditions (\ref{item : equiv free skel : 1}) and (\ref{item : equiv free skel : 2}) in the statement, as well as the following condition:
\benroman
\myitem{iii*}\label{item : equiv free skel : 3star} for all $a \in \Jirr(\A)$ and nonempty $Y \subseteq \At(\A)$,
\[
\At(\A) \cap \down a \subseteq Y \text{ implies } a \leq c_{1,Y}.
\] 
\eroman
Therefore, it only remains to show that $1 \in \Jirr(\A)$ and conditions (\ref{item : equiv free skel : 1}), (\ref{item : equiv free skel : 2}), and (\ref{item : equiv free skel : 3star}) hold if and only if $\A$ is nontrivial and conditions (\ref{item : equiv free skel : 1}), (\ref{item : equiv free skel : 2}), and (\ref{item : equiv free skel : 3}) hold. 

First, suppose that $1 \in \Jirr(\A)$ and conditions (\ref{item : equiv free skel : 1}), (\ref{item : equiv free skel : 2}), and (\ref{item : equiv free skel : 3star}) hold. From $1 \in \Jirr(\A)$ it follows that $1 \ne 0$. Hence, $\A$ is nontrivial. To show that (\ref{item : equiv free skel : 3}) holds, consider $b \in A - \{ 0 \}$. We prove that $\neg \neg b \in \Jirr(\A)$. Since $\A$ is finite and $b \ne 0$, the set $Y \coloneqq  \At(\A) \cap \down b$ is nonempty. Moreover, by Claim~\ref{Claim : properties double neg and atoms} we have $\At(\A) \cap \down \neg \neg b = \At(\A) \cap \down b = Y$. Thus, for every $a \in \Jirr(\A)$ with $a \leq \neg \neg b$ we have $\At(\A) \cap \down a \subseteq \At(\A) \cap \down \neg \neg b = Y$. Then (\ref{item : equiv free skel : 3star}) yields that $a \leq c_{1,Y}$ for every $a \in \Jirr(\A)$ such that $a \leq \neg \neg b$. Since $\A$ is finite, $\neg \neg b$ is the join of the join-irreducibles below it and, therefore, $\neg \neg b \leq c_{1,Y}$. By (\ref{item : equiv free skel : 1}) we have that $\At(\A) \cap \down c_{1,Y} = Y$. Since $Y = \At(\A) \cap \down b$, Claim~\ref{Claim : properties double neg and atoms} yields  $c_{1,Y} \leq \neg \neg b$. It follows that $\neg \neg b = c_{1,Y} \in \Jirr(\A)$. Hence, (\ref{item : equiv free skel : 3}) holds.

Conversely, assume that $\A$ is nontrivial and (\ref{item : equiv free skel : 1}), (\ref{item : equiv free skel : 2}), and (\ref{item : equiv free skel : 3}) hold. We need to prove that $1 \in \Jirr(\A)$ and (\ref{item : equiv free skel : 3star}) holds. As $\A$ is nontrivial,  we have $1 \in A - \{ 0 \}$. Together with $1= \neg \neg 1$ and (\ref{item : equiv free skel : 3}), this yields $1 \in \Jirr(\A)$. Now, for each nonempty $Y \subseteq \At(\A)$ define $c_{1,Y} \coloneqq \neg \neg \bigvee Y$, which is join-irreducible by (\ref{item : equiv free skel : 3}). To conclude the proof, it suffices to show that the elements $c_{1, Y}$ defined in this way satisfy (\ref{item : equiv free skel : 1}), (\ref{item : equiv free skel : 2}), and (\ref{item : equiv free skel : 3star}).

First, by Claim~\ref{Claim : properties double neg and atoms} we have $\At(\A) \cap \down c_{1,Y} = \At(\A) \cap \down \bigvee Y$.  Since atoms of distributive lattices are join-prime and $Y$ is a set of atoms, $\At(\A) \cap \down \bigvee Y = Y$. Thus, $\At(\A) \cap \down c_{1,Y} = Y$, whence~(\ref{item : equiv free skel : 1}) holds for $c_{1,Y}$. Then let $Z,Y \subseteq \At(\A)$ be nonempty. If $Y \subseteq Z$, then $\At(\A) \cap \down c_{1,Y} = Y \subseteq Z = \At(\A) \cap \down c_{1,Z}$. Then   Claim~\ref{Claim : properties double neg and atoms} implies $c_{1,Y} \leq \neg \neg c_{1,Z}$. Since $c_{1,Z}=\neg \neg \bigvee Z$, we have $\neg \neg c_{1,Z} = c_{1,Z}$ because $\neg \neg \neg b = \neg b$ for every $b \in A$ (see Proposition~\ref{Prop : triple negation}). Thus, $c_{1,Y} \leq c_{1,Z}$ and, therefore, condition (\ref{item : equiv free skel : 2}) holds for the elements of the form $c_{1,Y}$ and $c_{1,Z}$. To show that (\ref{item : equiv free skel : 3star}) holds, consider $a \in \Jirr(\A)$ and a nonempty subset $Y$ of $\At(\A)$ such that $\At(\A) \cap \down a \subseteq Y$. We show that $a \leq c_{1,Y}$. From (\ref{item : equiv free skel : 1}) it follows that $Y=\At(\A) \cap \down c_{1,Y}$. So,  $\At(\A) \cap \down a \subseteq \At(\A) \cap \down c_{1,Y}$, which implies $a \leq \neg \neg c_{1,Y}$ by Claim~\ref{Claim : properties double neg and atoms}. 
From the definition of $c_{1,Y}$ and Proposition~\ref{Prop : triple negation} it follows that $\neg \neg c_{1,Y} = c_{1,Y}$. Thus, $a \leq c_{1,Y}$. Then  (\ref{item : equiv free skel : 3star}) holds.
\end{proof}

\section{Models of the universal theory}

The aim of this section is to establish the following description of the models of the universal theory  $\mathsf{Th}_\forall(\boldsymbol{F}_\pDL(\aleph_0))$ of the free pseudocomplemented distributive lattice $\boldsymbol{F}_\pDL(\aleph_0)$.

\begin{Theorem}\label{Thm : MAIN}
The class of models of $\mathsf{Th}_\forall(\boldsymbol{F}_\pDL(\aleph_0))$ is
\[
\{ \A \in \pDL : \B_\ast\text{ has a free skeleton for every finite subalgebra }\B\text{ of }\A\}.
\]
\end{Theorem}

The above result is a consequence of the following fact, as we proceed to explain.

\begin{Proposition}\label{Prop : the main proposition}
Let $\A \in \pDL$ be finite. Then $\A$ embeds into $\boldsymbol{F}_\pDL(\aleph_0)$ if and only if the poset $\A_\ast$ has a free skeleton.
\end{Proposition}

 For suppose that Proposition \ref{Prop : the main proposition} holds. Then Theorem \ref{Thm : MAIN} can be derived as follows.

\begin{proof}
Recall from Theorem \ref{Thm : PDL locally finite} that $\pDL$ is a locally finite variety. Moreover, $\boldsymbol{F}_\pDL(\aleph_0) \in \pDL$ because $\pDL$ contains free algebras. Therefore, we can apply Theorem \ref{Thm : class of models of universal theory : general}, obtaining that the class of models of $\mathsf{Th}_\forall(\boldsymbol{F}_\pDL(\aleph_0))$ is
\[
\{ \A \in \pDL :  \text{each finite subalgebra of $\A$ embeds into $\boldsymbol{F}_\pDL(\aleph_0)$}\}.
\]
By Proposition \ref{Prop : the main proposition} this class coincides with the one in the statement of Theorem \ref{Thm : MAIN}.
\end{proof}

\begin{Remark}
A member $\A$ of a variety $\mathsf{V}$ is called \emph{exact} in $\mathsf{V}$ if it is isomorphic to a finitely generated subalgebra of $\boldsymbol{F}_{\mathsf{V}}(\aleph_0)$. Exact algebras serve as a fundamental tool in the study of exact unification \cite{CM15b}. As $\pDL$ is locally finite, the exact algebras in $\pDL$ are the finite pseudocomplemented distributive lattices that embed into $\boldsymbol{F}_\pDL(\aleph_0)$. It then follows from Proposition~\ref{Prop : the main proposition} that $\A \in \pDL$ is exact in $\pDL$ if and only if it is finite and $\A_\ast$ has a free skeleton.
\qed
\end{Remark}

The rest of this section is devoted to proving Proposition \ref{Prop : the main proposition}. To this end, it is convenient to fix some notation. Let $n \in \mathbb{Z}^+$. In view of Remark \ref{Rem : finite free extensions} and Corollary \ref{Cor : free pDL power of two}, we will identify the dual of the free algebra $\boldsymbol{F}_\pDL(n)$ with the finite poset $\P(\mathsf{2}^n)$ with universe
\[
\{ \langle x, C \rangle \in \mathsf{2}^n \times \wp(\mathsf{2}^n) : \varnothing \ne C \subseteq \up x \},
\]
ordered as follows:
\[
\langle x, C \rangle \leq \langle y, D \rangle \iff x \leq y \ \text{and} \ C \supseteq D.
\]

We begin with the following observations.

\begin{Proposition}\label{Prop : about Fn}
Let $n \geq 2$ and let $0$ be the minimum of $\mathsf{2}^n$. The following conditions hold:
\benroman
\item\label{item : basic structure of P(X) : 1} the minimum of $\P(\mathsf{2}^n)$ is $\langle 0, \mathsf{2}^n\rangle$;
\item\label{item : basic structure of P(X) : 2} the set of atoms of $\P(\mathsf{2}^n)$ is $\{ \langle 0, \mathsf{2}^n - \{ x \} \rangle : x \in \mathsf{2}^n\}$;
\item\label{item : basic structure of P(X) : 3} the set of maximal elements of $\P(\mathsf{2}^n)$ is $\{ \langle x, \{ x \} \rangle : x \in \mathsf{2}^n \}$.\ Furthermore, no atom is maximal.
\eroman
\end{Proposition}

\begin{proof}
(\ref{item : basic structure of P(X) : 1}): Immediate from the definition of $\P(\mathsf{2}^n)$. 

(\ref{item : basic structure of P(X) : 2}): We begin by proving that $\langle 0, \mathsf{2}^n - \{ x \} \rangle$ is an atom for every $x \in \mathsf{2}^n$. Consider $x\in \mathsf{2}^n$. Since $n \geq 2$ we have $\mathsf{2}^n - \{ x \} \ne \varnothing$.\ Together with $\mathsf{2}^n - \{ x \} \subseteq \mathsf{2}^n = \up 0$, this yields that $\langle 0, \mathsf{2}^n - \{ x \} \rangle$ is an element of $\P(\mathsf{2}^n)$. Moreover, $\langle 0, \mathsf{2}^n\rangle < \langle 0, \mathsf{2}^n - \{ x \} \rangle$ by the definition of $\P(\mathsf{2}^n)$. Since $\langle 0, \mathsf{2}^n\rangle$ is the minimum of $\P(\mathsf{2}^n)$ by (\ref{item : basic structure of P(X) : 1}), it only remains to prove that there exists no element $\langle z, Z \rangle \in \P(\mathsf{2}^n)$ such that $\langle 0, \mathsf{2}^n\rangle < \langle z, Z \rangle < \langle 0, \mathsf{2}^n - \{ x \} \rangle$. Suppose the contrary. By the definition of the order relation of $\P(\mathsf{2}^n)$ this yields
\[
0 \leq z \leq 0 \qquad \text{and} \qquad \mathsf{2}^n - \{ x \} \subseteq Z \subseteq \mathsf{2}^n.
\]
Hence, $z = 0$ and $Z \in \{ \mathsf{2}^n - \{ x \}, \mathsf{2}^n \}$, a contradiction with the assumption that $\langle z, Z \rangle$ differs from $\langle 0, \mathsf{2}^n\rangle$ and $\langle 0, \mathsf{2}^n - \{ x \} \rangle$. This concludes the proof that $\langle 0, \mathsf{2}^n - \{ x \} \rangle$ is an atom for each $x \in \mathsf{2}^n$.

Therefore, it only remains to prove that every atom is of this form. To this end, consider an atom $\langle x,C \rangle$. Observe that $C \ne \mathsf{2}^n$, otherwise from $\mathsf{2}^n = C \subseteq \up x$ it follows that $x = 0$ and, therefore, $\langle x,C \rangle = \langle 0,\mathsf{2}^n \rangle$, where the latter is the minimum $\P(\mathsf{2}^n)$ by (\ref{item : basic structure of P(X) : 1}). Then there exists $x \in \mathsf{2}^n-C$. So, $\langle 0, \mathsf{2}^n - \{ x \} \rangle$ is an element of $\P(\mathsf{2}^n)$ that is not the minimum and is below $\langle x,C \rangle$. Since $\langle x,C \rangle$ is an atom, we conclude that $\langle x,C \rangle = \langle 0, \mathsf{2}^n - \{ x \} \rangle$ as desired.

(\ref{item : basic structure of P(X) : 3}): We begin by showing that the set of maximal elements of $\P(\mathsf{2}^n)$ is $\{ \langle x, \{ x \} \rangle : x \in \mathsf{2}^n \}$. Clearly, $\langle x, \{ x \} \rangle \in \P(\mathsf{2}^n)$ for each $x \in \mathsf{2}^n$.\ Furthermore, if $x, y \in \mathsf{2}^n$ are distinct, then the elements $\langle x, \{ x \} \rangle$ and $\langle y, \{ y \} \rangle$ are incomparable because $x \notin \{ y\}$ and $y \notin \{ x \}$. Therefore, it suffices to show that every element of  $\P(\mathsf{2}^n)$ is below an element of the form $\langle x, \{ x \} \rangle$ for some $x \in \mathsf{2}^n$. Let $\langle x,C \rangle \in \P(\mathsf{2}^n)$. Since $C \ne \varnothing$, there exists $y \in C$. Moreover, $y \in C \subseteq \up x$ because $\langle x, C \rangle \in \P(\mathsf{2}^n)$. Hence, we conclude that $\langle x,C \rangle \leq \langle y, \{y\} \rangle$.

It only remains to show that no atom is maximal. By  (\ref{item : basic structure of P(X) : 2}) atoms are of the form $\langle 0, \mathsf{2}^n - \{ x \} \rangle$ with $x \in \mathsf{2}^n$. Then let $x \in \mathsf{2}^n$. Observe that $\mathsf{2}^n-\{x\} \ne \{y\}$ for every $y \in \mathsf{2}^n$ because $n \geq 2$ by assumption. Together with the above description of maximal elements, this yields that $\langle 0, \mathsf{2}^n - \{ x \} \rangle$ is not maximal.
\end{proof}

\begin{Proposition}\label{Prop : natural surjection : duals of free algebras}
For all $n, m \in \mathbb{Z}^+$ such that $n \leq m$ there exists a surjective weak p-morphism $p \colon \P(\mathsf{2}^m) \to \P(\mathsf{2}^n)$.
\end{Proposition}

\begin{proof}
By Proposition \ref{Prop : free algebra : technical prop}(\ref{item : 2 : free algebra : technical prop}) there exists an embedding $h \colon \boldsymbol{F}_\pDL(n) \to \boldsymbol{F}_\pDL(m)$. Therefore, we can apply Proposition \ref{Prop : weak p-morphic image} to obtain the desired surjective weak p-morphism $p \colon \P(\mathsf{2}^m) \to \P(\mathsf{2}^n)$.
\end{proof}

We are now ready to prove Proposition \ref{Prop : the main proposition}. In view of Corollary \ref{Cor : fin gen free algebras cover everything}, a finite pseudocomplemented distributive lattice $\A$ embeds into $\boldsymbol{F}_\pDL(\aleph_0)$ if and only if it embeds into $\boldsymbol{F}_\pDL(n)$ for some $n \in \mathbb{Z}^+$. Moreover, the latter is equivalent to the demand that $\A_\ast$ is a weak p-morphic image of $\P(\mathsf{2}^n)$ for some $n \in \mathbb{Z}^+$ by Proposition \ref{Prop : weak p-morphic image}. As a consequence, Proposition \ref{Prop : the main proposition} follows immediately from the next observation.

\begin{Proposition}
Let $X$ be a finite poset. Then $X$ is a weak p-morphic image of $\P(\mathsf{2}^n)$ for some $n \in \mathbb{Z}^+$ if and only if $X$ has a free skeleton.
\end{Proposition}

\begin{proof}
We begin by proving the implication from left to right. Let $p \colon \P(\mathsf{2}^n) \to X$ be a surjective weak p-morphism with $n \in \mathbb{Z}^+$.\ First, we prove that $X$ has a minimum. Recall from Proposition \ref{Prop : about Fn}(\ref{item : basic structure of P(X) : 1}) that $\langle 0, \mathsf{2}^n \rangle$ is the minimum of $\P(\mathsf{2}^n)$, where $0$ is the minimum of $\mathsf{2}^n$.\ We will show that $\bot \coloneqq p(\langle 0, \mathsf{2}^n \rangle)$ is the minimum of $X$. To this end, let $x \in X$. Since $p \colon \P(\mathsf{2}^n) \to X$ is surjective, there exists $y \in \P(\mathsf{2}^n)$ such that $p(y) = x$. Moreover, $\langle 0, \mathsf{2}^n \rangle \leq y$ because $\langle 0, \mathsf{2}^n \rangle$ is the minimum of $\P(\mathsf{2}^n)$.\ As $p$ is order preserving, we conclude that $\bot = p(\langle 0, \mathsf{2}^n \rangle) \leq p(y) = x$. Hence, $\bot$ is the minimum of $X$ as desired.

Now, observe that for each $x \in \mathsf{2}^n$ the pair $\langle x, \{ x \} \rangle$ is an element of $\P(\mathsf{2}^n)$. Then for each nonempty $Y \subseteq \max X$ let
\[
m_Y \coloneqq \{ x \in \mathsf{2}^n : p(\langle x, \{ x\} \rangle) \in Y \}.
\]
\begin{Claim}\label{Claim : 5.4}
For all $x \in X$, nonempty $Y \subseteq \max \up x$, and $\langle z, Z \rangle \in \P(\mathsf{2}^n)$,
\[
\text{if }p(\langle z, Z \rangle) = x, \text{ then }\langle z, Z \cap m_Y \rangle \in \P(\mathsf{2}^n).
\]
\end{Claim}

\begin{proof}[Proof of the Claim]
Suppose that $p(\langle z, Z \rangle) = x$. From $\langle z, Z \rangle \in \P(\mathsf{2}^n)$ it follows that $Z \subseteq \up z$. Then $Z \cap m_Y \subseteq \up z$ as well. Therefore, it only remains to prove that $Z \cap m_Y \ne \varnothing$. Since $Y \ne  \varnothing$, there exists $y \in Y \subseteq \max \up x$. By Proposition~\ref{Prop : preservation of maximal}(\ref{item : preservation of maximal : 1}) we have
\[
y \in \max \up x = \max \up p(\langle z, Z \rangle) = p[\max \up \langle z, Z \rangle].
\]
Together with Proposition \ref{Prop : about Fn}(\ref{item : basic structure of P(X) : 3}), this implies that there exists $v \in \mathsf{2}^n$ such that $\langle z, Z \rangle \leq \langle v, \{ v \} \rangle$ and $p(\langle v, \{ v \} \rangle) = y$.\ From $\langle z, Z \rangle \leq \langle v, \{ v \} \rangle$ it follows that $v \in Z$ and from $p(\langle v, \{ v \} \rangle) = y$ that $v \in m_Y$. Thus, $v \in Z \cap m_Y$ and, therefore, $Z \cap m_Y \ne \varnothing$.
\end{proof}

Recall that the minimum $\bot$ of $X$ coincides with $p(\langle 0, \mathsf{2}^n\rangle)$. Therefore,  Claim \ref{Claim : 5.4} guarantees that for each nonempty $Y \subseteq \max X$ we have $\langle 0, m_Y \rangle = \langle 0, \mathsf{2}^n \cap m_Y \rangle\in \P(\mathsf{2}^n)$. Then we define $s_{\bot, Y} \coloneqq p(\langle 0, m_Y\rangle)$. Furthermore, since $p$ is surjective, for every $x \in X - \{ \bot \}$ there exists $\langle z, Z \rangle \in \P(\mathsf{2}^n)$ such that $p(\langle z, Z\rangle) = x$. By the Claim we have $\langle z, Z \cap m_Y \rangle\in \P(\mathsf{2}^n)$ for every nonempty $Y \subseteq \max \up x$. Then for every nonempty $Y \subseteq \max \up x$ we define $s_{x, Y} \coloneqq p(\langle z, Z \cap m_Y\rangle)$.

We will prove that the elements $s_{\bot, Y}$ and $s_{x, Y}$ defined in this way witness the fact that $X$ has a free skeleton. We begin by proving condition (\ref{item : def : 1}) of Definition \ref{def}. Let $x \in X$ and let $Y \subseteq \max \up x$ be nonempty. We need to prove that $s_{x, Y} \in \up x$ and $Y = \max \up s_{x, Y}$. First, by the definition of $s_{x, Y}$ there exists $\{z\} \cup Z \subseteq \mathsf{2}^n$ such that $s_{x, Y} = p(\langle z, Z \cap m_Y\rangle)$ and $x = p(\langle z, Z \rangle)$.  Therefore, from $\langle z, Z \rangle \leq \langle z, Z \cap m_Y\rangle$ and the fact that $p$ is order preserving it follows that $x = p(\langle z, Z \rangle) \leq p(\langle z, Z \cap m_Y\rangle) = s_{x, Y}$. Then we prove that $Y = \max \up s_{x, Y}$. 
 By Proposition~\ref{Prop : preservation of maximal}(\ref{item : preservation of maximal : 1}) we have
\[
\max \up s_{x, Y} = \max \up p(\langle z, Z \cap m_Y\rangle) = p[\max \up \langle z, Z \cap m_Y\rangle].
\]
Moreover, by Proposition \ref{Prop : about Fn}(\ref{item : basic structure of P(X) : 3}) and the definition of the order relation of $\P(\mathsf{2}^n)$ we have
\[
\max \up \langle z, Z \cap m_Y\rangle = \{ \langle v, \{ v \} \rangle : v \in \mathsf{2}^n \text{ and }v \in Z \cap m_Y \}.
\]
In view of the two displays above, it only remains to show that 
\[
Y = p[\{ \langle v, \{ v \} \rangle : v \in \mathsf{2}^n \text{ and }v \in Z \cap m_Y \}].
\]

To prove the inclusion from left to right, consider $y \in Y \subseteq \max \up x$. Since $x = p(\langle z, Z \rangle)$, we can apply Proposition~\ref{Prop : preservation of maximal}(\ref{item : preservation of maximal : 1}), obtaining $y \in p[\max \up \langle z, Z \rangle]$. By Proposition \ref{Prop : about Fn}(\ref{item : basic structure of P(X) : 3}) there exists $v \in \mathsf{2}^n$ such that $\langle z, Z \rangle \leq \langle v, \{ v \} \rangle$ and $p(\langle v, \{ v \} \rangle) = y$. From the definition of the order relation of $\P(\mathsf{2}^n)$ and of $m_Y$ it follows that 
$v \in Z \cap m_Y$.
Together with $p(\langle v, \{ v \} \rangle) = y$, this yields that $y$ belongs to the right hand side of the above display. Then we prove the inclusion from right to left. Let $y$ be an element of the right hand side of the above display. Then there exists $v \in \mathsf{2}^n$ such that 
$v \in Z \cap m_Y$
and $p(\langle v, \{ v \} \rangle) = y$. Since $v \in m_Y$, we conclude that $y \in Y$ as desired. This concludes the proof of condition (\ref{item : def : 1}) of Definition \ref{def}.

Then we turn our attention to proving condition (\ref{item : def : 2}).\ Let $x \in X$ and consider a pair of nonempty sets $Y_1, Y_2 \subseteq \max \up x$ such that $Y_1 \subseteq Y_2$. Then $s_{x, Y_1} = p(\langle z, Z \cap m_{Y_1}\rangle)$ and $s_{x, Y_2} = p(\langle z, Z \cap m_{Y_2}\rangle)$ for some $\{z\} \cup Z \subseteq \mathsf{2}^n$. Since $Y_1 \subseteq Y_2$, the definition of $m_{Y_1}$ and $m_{Y_2}$ guarantees that $m_{Y_1} \subseteq m_{Y_2}$. Therefore, $\langle z, Z \cap m_{Y_2}\rangle \leq \langle z, Z \cap m_{Y_1}\rangle$ by the definition of the order relation of $\P(\mathsf{2}^n)$. As $p$ is order preserving, we conclude that $s_{x, Y_2} = p(\langle z, Z \cap m_{Y_2}\rangle) \leq p(\langle z, Z \cap m_{Y_1}\rangle) = s_{x, Y_1}$.

It only remains to prove condition (\ref{item : def : 3}). Consider $x \in X$ and a nonempty set $Y \subseteq \max X$ such that $\max \up x \subseteq Y$. Since $p$ is surjective, there exists $\langle y, Z \rangle \in \P(\mathsf{2}^n)$ such that $x = p(\langle y, Z \rangle)$. We will prove that $Z \subseteq m_Y$. To this end, consider $z \in Z$. From $\langle y, Z \rangle \in \P(\mathsf{2}^n)$ and $z \in Z$ it follows that $y \leq z$.\ Together with $z \in Z$, this implies $\langle y, Z \rangle \leq \langle z, \{ z \} \rangle$.\ Recall from Proposition \ref{Prop : about Fn}(\ref{item : basic structure of P(X) : 3}) that $\langle z, \{ z \} \rangle \in \max \up \langle y, Z \rangle$. Therefore, we can apply Proposition~\ref{Prop : preservation of maximal}(\ref{item : preservation of maximal : 1}), obtaining that $p(\langle z, \{ z \} \rangle) \in \max \up p(\langle y, Z \rangle) = \max \up x$. As $\max \up x \subseteq Y$ by assumption, we have $p(\langle z, \{ z \} \rangle) \in Y$. By the definition of $m_Y$ we conclude that $z \in m_Y$ as desired. This establishes that $Z \subseteq m_Y$. Together with $0 \leq z$, this yields $\langle 0, m_Y \rangle \leq \langle z, Z \rangle$. As $p$ is order preserving, we obtain $p(\langle 0, m_Y \rangle) = p(\langle z, Z \rangle) = x$. Lastly, recall that $s_{\bot, Y} = p(\langle 0, m_Y \rangle)$ by definition. Hence, $s_{\bot, Y} \leq x$. This concludes the proof that $X$ has a free skeleton.

Then we prove the implication from right to left. Let $X$ be a finite poset with a free skeleton. We will reason by induction on $\lvert \max X \rvert$. Notice that $\lvert \max X \rvert \geq 1$ because $X$ is finite and nonempty by assumption.

\subsection*{Base case} In this case, $\lvert \max X \rvert = 1$. Therefore, $X$ has a maximum $\top$. Furthermore, $X$ has a minimum $\bot$ because it has a free skeleton. We will construct a surjective weak p-morphism $p \colon \P(\mathsf{2}^n) \to X$ for $n = \lvert X \rvert+1$.\ To this end, observe that $n \geq 2$ because $X \ne \varnothing$. By Proposition \ref{Prop : about Fn} the poset $\P(\mathsf{2}^n)$ has a minimum $\bot_n$ and $2^n$ atoms, none of which is maximal. Let $Y$ be the set of atoms of $\P(\mathsf{2}^n)$ and observe that $\lvert X \rvert = n+1 \leq 2^n = \lvert Y \rvert$. Then let $p \colon \P(\mathsf{2}^n) \to X$ be any map satisfying the following requirements:
\benroman
\item\label{item : 1 : map p : base case : hard part} $p(\bot_n) = \bot$;
\item\label{item : 2 : map p : base case : hard part} $p$ restricts to a surjection from $Y$ to $X$;
\item\label{item : 3 : map p : base case : hard part} $p(x) = \top$ for every $x \in \P(\mathsf{2}^n) - (Y \cup \{ \bot_n \})$.
\eroman

First, $p$ is surjective by (\ref{item : 2 : map p : base case : hard part}). Then we prove that $p$ is order preserving. Consider $x, y \in \P(\mathsf{2}^n)$ such that $x < y$. We have two cases: either $x = \bot_n$ or $\bot_n < x$. If $x = \bot_n$, then $p(x) = \bot \leq p(y)$ by (\ref{item : 1 : map p : base case : hard part}). Then we consider the case where $\bot_n < x$. Together with $x < y$, this implies that $y \in \P(\mathsf{2}^n) - (Y \cup \{ \bot_n \})$. Therefore, $p(x) \leq \top = p(y)$ by (\ref{item : 3 : map p : base case : hard part}). Hence, we conclude that $p$ is order preserving.

It only remains to prove that $p$ satisfies the weak p-morphism condition.\ To this end, consider $x \in \P(\mathsf{2}^n)$ and $y \in \max \up p(x)$. Since $\max \up p(x) \subseteq \max X$ and by assumption $\max X = \{ \top \}$, we obtain $y = \top$. Moreover, as $\P(\mathsf{2}^n)$ is finite, has atoms, and none of its atoms is maximal, there exists $z \in \P(\mathsf{2}^n) - (Y \cup \{ \bot_n \})$ such that
 $z \in \max \up x$. By (\ref{item : 3 : map p : base case : hard part}) we conclude that $p(z) = \top = y$. Hence, $p \colon \P(\mathsf{2}^n) \to X$ is a surjective weak p-morphism as desired.

\subsection*{Inductive step} In this case, $\lvert \max X \rvert = n+1$ for some $n \geq 1$. Fix the enumerations
\begin{equation}\label{Eq : the first assumption of the inductive step}
\max X = \{ w_1, \dots, w_{n+1} \} \, \, \text{ and } \, \,
\max \{ x \in X : x \leq w_1, \dots, w_{n+1} \} = \{ v_1, \dots, v_t \}.
\end{equation}
Since $X$ has a free skeleton, it has a minimum. Hence, the second set above is nonempty. For each $i \leq n+1$ let
\[
M_i \coloneqq (\max X) - \{ w_i \}.
\]
Moreover, let $\bot$ be the minimum of $X$. We will prove that
\begin{equation}\label{Eq : the partition of X}
X = {\uparrow}s_{\bot, M_1} \cup \dots \cup {\uparrow}s_{\bot, M_{n+1}} \cup {\downarrow}\{ v_1, \dots, v_t \}.
\end{equation}
Clearly, it suffices to prove the inclusion from left to right. Let $x \in X$. If $\max \up x \subseteq M_i$ for some $i \leq n+1$, then $x \in {\uparrow}s_{\bot, M_i}$ by condition~(\ref{item : def : 3}) of Definition~\ref{def}. Otherwise, $\max \up x = \max X$, and so $x \leq w_1, \dots, w_{n+1}$. Thus, $x \leq v_i$ for some $i \leq t$.

By Proposition \ref{Prop : skeletal upsets} each poset ${\uparrow}s_{\bot, M_i}$ has a free skeleton.\ Furthermore, ${\uparrow}s_{\bot, M_i}$ has $n$ maximal elements by construction (namely, the elements of $M_i$).\ Therefore, we can apply the inductive hypothesis, obtaining a surjective weak p-morphism $p_i \colon \P(\mathsf{2}^{m_i}) \to {\uparrow}s_{\bot, M_i}$ for some $m_i \in \mathbb{Z}^+$. By letting $m \coloneqq \max \{ m_1, \dots, m_{n+1} \}$ and invoking Proposition \ref{Prop : natural surjection : duals of free algebras}, we may assume that the domain of each $p_i$ is $\P(\mathsf{2}^{m})$.  Therefore, from now on we will work with surjective weak p-morphisms $p_i \colon \P(\mathsf{2}^{m}) \to {\uparrow}s_{\bot, M_i}$.

While the following observation on Boolean lattices is an easy exercise, we decided to sketch a proof for the sake of completeness.

\begin{Claim}\label{Claim : upsets}
There exists $k \in \mathbb{Z}^+$ such that $\mathsf{2}^k$ contains distinct elements 
\[
\hat{w}_1, \dots, \hat{w}_{n+1}, \hat{v}_1, \dots, \hat{v}_t
\]
satisfying the following conditions:
\benroman
\item\label{item : upsets : 1} if $x_1,x_2 \in \{\hat{w}_1, \dots, \hat{w}_{n+1}, \hat{v}_1, \dots, \hat{v}_t\}$ are distinct, then ${\uparrow}x_1 \cap {\uparrow} x_2 = \{ 1 \}$, where $1$ is the maximum of $\mathsf{2}^k$;
\item\label{item : upsets : 2} for each $i \leq n+1$ there exists a coatom $a_i \in \mathsf{2}^k$ such that $[\hat{w}_i, a_i] \cong \mathsf{2}^m$;
\item\label{item : upsets : 3} for each $j \leq t$ there exist coatoms $b^j_1, \dots, b^j_{n+1} \in \mathsf{2}^k$ such that $[\hat{v}_j, b^j_1 \land \dots \land b^j_{n+1}] \cong \mathsf{2}^{\lvert {\downarrow}v_j \rvert+1}$.
\eroman
\end{Claim}

\begin{proof}[Proof of the Claim]
Let 
\[
k \coloneqq (n+1)(m+1)+ \sum_{j=1}^t (n+1+\lvert \down v_j\rvert+1)
\]
and identify the elements of $\mathsf{2}^k$ with the functions $f \colon \{1, \dots, k\} \to \{0,1\}$. 
Moreover, consider $i \leq n+1$ and $j \leq t$. Then let $\hat{w}_i, \hat{v}_j, a_i, b^j_1, \dots, b^j_{n+1}$ be the functions from $\{1, \dots, k\}$ to $\{0,1\}$ defined as follows: for each $h \in \{1, \dots, k\}$,
\begin{align*}
\hat{w}_i(h) = 0 & \iff (i-1)(m+1) < h \leq i(m+1);\\
\hat{v}_j(h) = 0 & \iff \sum_{s=1}^{j-1} (n+1+\lvert \down v_s\rvert+1) < h - (n+1)(m+1) \leq \sum_{s=1}^{j} (n+1+\lvert \down v_s\rvert+1);\\
a_i(h) = 0 & \iff h=i(m+1);\\
b^j_{l}(h) = 0 & \iff h = (n+1)(m+1) + \sum_{s=1}^{j} (n+1+\lvert \down v_s\rvert+1) - (l-1).
\end{align*}
It is straightforward to check that conditions (\ref{item : upsets : 1}), (\ref{item : upsets : 2}), and (\ref{item : upsets : 3}) hold.
\end{proof}

Fix $k \in \mathbb{Z}^+$ and $\hat{w}_i,\hat{v}_j, a_i, b^j_1, \dots, b^j_{n+1} \in \mathsf{2}^k$ for all $i \leq n+1$ and $j \leq t$ satisfying conditions~(\ref{item : upsets : 1}),~(\ref{item : upsets : 2}), and~(\ref{item : upsets : 3}) of Claim~\ref{Claim : upsets}. Let $i \leq n+1$ and $j \leq t$. By condition (\ref{item : upsets : 2})  of Claim~\ref{Claim : upsets} we have $\hat{w}_i \leq a_i$, whence $\varnothing \ne [\hat{w}_i, a_i] \subseteq \up \hat{w}_i$. Similarly, by condition (\ref{item : upsets : 3})  of Claim~\ref{Claim : upsets} we have $\hat{v}_j \leq b_1^j, \dots, b_{n+1}^j$, whence $\varnothing \ne  [\hat{v}_j, b^j_1 \land \dots \land b^j_{n+1}] \cup \{ b_1^j, \dots, b_{n+1}^j\} \subseteq \up \hat{v}_j$. Therefore,
\[
\langle \hat{w}_i, [\hat{w}_i, a_i]\rangle,  \langle \hat{v}_j, [\hat{v}_j, b^j_1 \land \dots \land b^j_{n+1}] \cup \{ b_1^j, \dots, b_{n+1}^j\} \rangle \in \P(\mathsf{2}^k).
\]
Then we consider the following upsets of $\P(\mathsf{2}^k)$: 
\[
W_i \coloneqq {\uparrow}\langle \hat{w}_i, [\hat{w}_i, a_i]\rangle \qquad \text{and} \qquad V_j \coloneqq {\uparrow}\langle \hat{v}_j, [\hat{v}_j, b^j_1 \land \dots \land b^j_{n+1}] \cup \{ b_1^j, \dots, b_{n+1}^j\} \rangle.
\]

\begin{Claim}\label{Claim : disjoint}
The upsets $W_1, \dots, W_{n+1}, V_1, \dots, V_{t}$ are pairwise disjoint.
\end{Claim}

\begin{proof}[Proof of the Claim]
Suppose, with a view to contradiction, that there are distinct
\begin{equation}\label{Eq : The sets U 1 and 2}
U_1, U_2 \in \{ W_1, \dots, W_{n+1}, V_1, \dots, V_{t}\}
\end{equation}
and some $\langle x, Z \rangle \in U_1\cap  U_2$.\ By the definition of the $W_i$'s and $V_j$'s and the fact that the elements $\hat{w}_1, \dots, \hat{w}_{n+1}, \hat{v}_1, \dots, \hat{v}_t$ are distinct by assumption there exist distinct 
\[
x_1, x_2 \in \{ \hat{w}_1, \dots, \hat{w}_{n+1}, \hat{v}_1, \dots, \hat{v}_t\}
\]
such that $U_1 = {\uparrow}\langle x_1, Z_1\rangle$ and $U_2 = {\uparrow}\langle x_2, Z_2\rangle$ for some $Z_1, Z_2 \subseteq \mathsf{2}^k$.\ As $\langle x, Z \rangle \in U_1\cap  U_2$, this yields $x_1, x_2 \leq x$. By condition (\ref{item : upsets : 1}) of Claim~\ref{Claim : upsets} we obtain $x = 1$, where $1$ is the maximum of $\mathsf{2}^k$.\ Since $\langle 1, Z \rangle = \langle x, Z \rangle \in \P(\mathsf{2}^k)$, we obtain $\varnothing \ne Z  \subseteq {\uparrow}1 = \{ 1 \}$. Therefore, $Z = \{ 1 \}$.\ Consequently, $\langle 1, \{ 1 \} \rangle = \langle x, Z \rangle \in U_1$. As $U_1 = {\uparrow}\langle x_1, Z_1\rangle$, this yields $1 \in Z_1$. By  (\ref{Eq : The sets U 1 and 2}) we have
\[
{\uparrow}\langle x_1, Z_1\rangle = U_1 \in \{ W_1, \dots, W_{n+1}, V_1, \dots, V_{t}\}.
\]
Therefore, the definition of the $W_i$'s and $V_j$'s guarantees that $Z_1$ is of the form $[\hat{w}_i, a_i]$ or $[\hat{v}_j, b^j_1 \land \dots \land b^j_{n+1}] \cup \{ b_1^j, \dots, b_{n+1}^j\}$ for some $i \leq n+1$ and $j \leq t$. As $a_i, b_1^j, \dots, b_{n+1}^j$ are coatoms of $\mathsf{2}^k$ by conditions (\ref{item : upsets : 2}) and (\ref{item : upsets : 3})  of Claim~\ref{Claim : upsets}, we conclude that $1 \notin Z_1$, a contradiction with $1 \in Z_1$.
\end{proof}

Now, for each $j \leq t$ let
\[
V_j^+ \coloneqq \{\langle x,Z\rangle \in V_j : x \in [\hat{v}_j, b^j_1 \land \dots \land b^j_{n+1}] \text{ and } \{ b_1^j, \dots, b_{n+1}^j \} \subsetneq Z  \}.
\]
From the definition of the order relation of $\P(\mathsf{2}^k)$ it immediately follows that $V_j^+$ is a downset of $V_j$.

\begin{Claim}\label{Claim : isomorphisms}
For all $i \leq n+1$ and $j \leq t$ we have
\[
W_i \cong \P(\mathsf{2}^m) \qquad \text{and} \qquad V_j^+ \cong \P(\mathsf{2}^{\lvert {\downarrow} v_j \rvert+1}).
\]
\end{Claim}

\begin{proof}[Proof of the Claim]
From the definition of $W_i$ it follows that $W_i=\P([\hat{w}_i, a_i])$. Moreover, the definition of $V_j^+$ implies that the map $\hat{f} \colon V_j^+ \to \P([\hat{v}_j, b^j_1 \land \dots \land b^j_{n+1}])$ given by $\hat{f}(\langle x,Z \rangle) = \langle x,Z-\{ b^j_1, \dots, b^j_{n+1}\} \rangle$ is an isomorphism. Since an isomorphism $X \cong Y$ between finite posets induces an isomorphism $\P(X) \cong \P(Y)$, conditions (\ref{item : upsets : 2}) and (\ref{item : upsets : 3}) of Claim~\ref{Claim : upsets} yield
\[
\P([\hat{w}_i, a_i]) \cong \P(\mathsf{2}^m) \qquad \text{and} \qquad \P([\hat{v}_j, b^j_1 \land \dots \land b^j_{n+1}]) \cong \P(\mathsf{2}^{\lvert {\downarrow} v_j \rvert+1}).
\]
Therefore, we conclude that
\[
W_i \cong \P(\mathsf{2}^m) \qquad \text{and} \qquad V_j^+ \cong \P(\mathsf{2}^{\lvert {\downarrow} v_j \rvert+1}).\qedhere
\]
\end{proof}

We will rely on the following observation.

\begin{Claim}\label{Claim : the two maps}
The following conditions hold for all $i \leq n+1$ and $j \leq t$:
\benroman
\item\label{item : claim : 1} there exists a weak p-morphism $\hat{p}_i \colon W_i \to X$ such that ${\uparrow}s_{\bot, M_i} \subseteq \hat{p}_i[W_i]$;
\item\label{item : claim : 2} there exists a weak p-morphism $\hat{q}_j \colon V_j \to X$ such that ${\downarrow}v_j \subseteq \hat{q}_j[V_j]$.
\eroman
\end{Claim}

\begin{proof}[Proof of the Claim]
(\ref{item : claim : 1}): Recall that the map $p_i \colon \P(\mathsf{2}^m) \to {\uparrow}s_{\bot, M_i}$ is a surjective weak p-morphism. In view of Claim \ref{Claim : isomorphisms}, we can identify $\P(\mathsf{2}^m)$ with $W_i$. Therefore, we can view $p_i$ as a surjective weak p-morphism $\hat{p}_i \colon W_i \to {\uparrow}s_{\bot, M_i}$. Lastly, as ${\uparrow}s_{\bot, M_i}$ is an upset of $X$, we may assume that the codomain of $\hat{p}_i$ is $X$. In this way, we obtain a weak p-morphism $\hat{p}_i \colon W_i \to X$ such that ${\uparrow}s_{\bot, M_i} \subseteq \hat{p}_i[W_i]$.

(\ref{item : claim : 2}): Observe that the poset ${\downarrow}v_j$ is finite  and has maximum and minimum (the latter because $X$ has a free skeleton). Therefore, ${\downarrow}v_j$ has also a free skeleton by Proposition \ref{Prop : bounded posets have a free skeleton} and only one maximal element. Consequently, we can apply the base case of the induction to ${\downarrow}v_j$, obtaining a surjective weak p-morphism $q_j \colon \P(\mathsf{2}^{\lvert {\downarrow}v_j \rvert +1 }) \to {\downarrow}v_j$. In view of Claim \ref{Claim : isomorphisms}, we can identify $\P(\mathsf{2}^{\lvert {\downarrow}v_j \rvert +1 })$ with $V_j^+$. Therefore, we can view $q_j$ as a surjective weak p-morphism $q_j \colon V_j^+ \to {\downarrow}v_j$.

We will extend $q_j$ to a weak p-morphism $\hat{q}_j \colon V_j \to X$ such that ${\downarrow}v_j \subseteq \hat{q}_j[V_j]$. To this end, recall that $\max X = \{ w_1, \dots, w_{n+1} \}$. For each $x \in V_j$ let
\[
d(x) \coloneqq
		  \{ w_i : x \leq \langle b_i^j, \{ b_i^j \} \rangle \}
\]
and
\[
e(x) \coloneqq  \begin{cases}
		  d(x)  & \text{ if } \max {\uparrow}x \subseteq \{ \langle b_1^j, \{ b_1^j \} \rangle, \dots, \langle b_{n+1}^j, \{ b_{n+1}^j \} \rangle \};\\
		d(x) \cup \{ w_1 \} & \text{ otherwise}.
		\end{cases}
\]
Notice that for each $x \in V_j$ the set $e(x)$ is a nonempty subset of $\max X = \{ w_1, \dots, w_{n+1} \}$. Furthermore, $\max X = \max{\uparrow}v_j$ by the definition of $v_j$. Therefore, for each $x \in V_j$ we can consider the element $s_{v_j, e(x)}$ of $X$. Then let $\hat{q}_j \colon V_j \to X$ be the map defined for every $x \in V_j$ as
\[
\hat{q}_j(x) \coloneqq  \begin{cases}
		 q_j(x) & \text{ if } x \in V_j^+ ;\\
		s_{v_j, e(x)} & \text{ if }x \notin V_j^+ \cup \max \P(\mathsf{2}^k);\\
				w_i & \text{ if }x \in \max \P(\mathsf{2}^k) \text{ and }e(x) = \{w_i\}.
		\end{cases}
\]

We begin by proving that $\hat{q}_j$ is well defined. First, the conditions in the definition of $\hat{q}_j$ are mutually exclusive because $V_j^+ \cap \max \P(\mathsf{2}^k) = \varnothing$ by the definition of $V_j^+$ and the description of $\max \P(\mathsf{2}^k)$ in Proposition \ref{Prop : about Fn}(\ref{item : basic structure of P(X) : 3}). Therefore, it suffices to show that they cover all the possible cases. To this end, consider $x \in V_j$ for which the first two conditions in the definition of $\hat{q}_j$ fail. Then $x \in \max \P(\mathsf{2}^k)$. Therefore, we can apply Proposition \ref{Prop : about Fn}(\ref{item : basic structure of P(X) : 3}) obtaining that $x$ can be below at most one element of the form $\langle b_i^j, \{ b_i^j \} \rangle$ for $i \leq n+1$, in which case $x = \langle b_i^j, \{ b_i^j \} \rangle$. By the definition of $d(x)$ this yields that one of the following conditions holds:
\benormal
\item[A.]\label{item : case 1 : e : well defined} $x \ne \langle b_i^j, \{ b_i^j \} \rangle$ for all $i \leq n+1$ and $d(x) = \emptyset$;
\item[B.] $x = \langle b_i^j, \{ b_i^j \}\rangle$ and $d(x) = \{ w_i \}$ for some $i \leq n+1$.
\enormal
Moreover, $\up x = \{ x \}$ because $x$ is maximal. Together with the definition of $e(x)$, this implies that if case (A) holds, then $e(x) = \{ w_1 \}$, while if case (B) holds, then $e(x) = \{ w_i \}$. In both cases, the third condition in the definition of $\hat{q}_j$ holds for $x$. Hence, we conclude that $\hat{q}_j$ is well defined as desired. 

From the definition of $\hat{q}_j$ it follows that ${\downarrow}v_j \subseteq \hat{q}_j[V_j]$ because $q_j \colon V_j^+ \to {\downarrow}v_j$ is surjective and $\hat{q}_j$ extends $q_j$. Therefore, it only remains to prove that $\hat{q}_j$ is a weak p-morphism.

We begin by proving that $\hat{q}_j$ is order preserving. To this end, let $x, y \in V_j$ be such that $x < y$. If $x, y \in V_j^+$, then $\hat{q}_j(x) = q_j(x) \leq q_j(y) = \hat{q}_j(y)$, where the middle inequality holds because $q_j$ is order preserving by assumption. Therefore, we may assume that either $x$ or $y$ does not belong to $V_j^+$. As $V_j^+$ is a downset of $V_j$, $x < y$, and $x, y \in V_j$, we obtain $y \notin V_j^+$.  By the definition of $\hat{q}_j$ this implies $\hat{q}_j(y) \in \{ s_{v_j, e(y)}, w_1, \dots, w_{n+1}\}$. Since $v_j \leq  s_{v_j, e(y)}$ and $v_j \leq w_1, \dots, w_{n+1}$ (the latter by (\ref{Eq : the first assumption of the inductive step})), we obtain $v_j \leq \hat{q}_j(y)$.
 Now, we have two cases: either $x \in V_j^+$ or $x \notin V_j^+$. 
First suppose that $x \in V_j^+$. Then $\hat{q}_j(x) = q_j(x)$. Since the codomain of $q_j$ 
is ${\downarrow}v_j$, we obtain $q_j(x) \leq v_j$ and, therefore, $\hat{q}_j(x) = q_j(x) \leq v_j \leq \hat{q}_j(y)$ as desired.  Then we consider the case where $x \notin V_j^+$. Since $x < y$ by assumption, we have $x \notin \max \P(\mathsf{2}^k)$. Therefore, $x \notin V_j^+ \cup \max \P(\mathsf{2}^k)$ and by the definition of $\hat{q}_j$ we obtain $\hat{q}_j(x) = s_{v_j, e(x)}$.\ We have two cases: either $\hat{q}_j(y) = s_{v_j, e(y)}$ or not. Suppose first that $\hat{q}_j(y) = s_{v_j, e(y)}$. From $x \leq y$ it follows that $d(y) \subseteq d(x)$. Together with $\max \up y \subseteq \max \up x$ and the definition of $e(x)$ and $e(y)$, this yields $e(y) \subseteq e(x)$. Therefore,  $s_{v_j, e(x)} \leq s_{v_j, e(y)}$ by condition (\ref{item : def : 2}) of Definition \ref{def}. Consequently, $\hat{q}_j(x) = s_{v_j, e(x)} \leq s_{v_j, e(y)} = \hat{q}_j(y)$ as desired. Then we consider the case where $\hat{q}_j(y) \ne s_{v_j, e(y)}$. Since $y \notin V_j^+$ by assumption, the definition of $\hat{q}_j$ guarantees that $\hat{q}_j(y) = w_i$ for some $i \leq n+1$ such that $e(y) = \{w_i\}$.\ As $e(y) \subseteq e(x)$, we get $w_i \in e(x)$. Then $s_{v_j, e(x)} \leq w_i$ by condition (\ref{item : def : 1}) of Definition \ref{def}. Hence, $\hat{q}_j(x) = s_{v_j, e(x)} \leq w_i = \hat{q}_j(y)$. This concludes the proof that $\hat{q}_j$ is order preserving.

Then we prove that $\hat{q}_j$ satisfies the weak p-morphism condition.\ Recall that $\max X = \{ w_1, \dots, w_{n+1} \}$. Then consider $x \in V_j$ and $w_i \in \max {\uparrow}\hat{q}_j(x)$. We need to find some $y \in \max{\uparrow}x$ such that $\hat{q}_j(y) = w_i$. We have two cases: either $x \in \max \P(\mathsf{2}^k)$ or $x \notin \max \P(\mathsf{2}^k)$. Suppose first that $x \in \max \P(\mathsf{2}^k)$. Recall that $V_j^+ \cap \max \P(\mathsf{2}^k) = \varnothing$. Therefore, $x \in \max \P(\mathsf{2}^k) - V_j^+$. Since the map $\hat{q}_j$ is well defined, the element $x$ should satisfy at least one of the three conditions in the definition of $\hat{q}_j$. From $x \in \max \P(\mathsf{2}^k) - V_j^+$ it follows that $x$ can only satisfy the third condition,
which means that there exists $h \leq n+1$ such that $e(x) = \{ w_h \}$ and, therefore, $\hat{q}_j(x) = w_h$. 
From $w_h = \hat{q}_j(x) \leq w_i$ and the fact that $w_h \in \max X$ it follows that $w_h = w_i$, whence $\hat{q}_j(x) = w_i$. Therefore, we are done letting $y \coloneqq x$. Then we consider the case where $x \notin \max \P(\mathsf{2}^k)$. We have two cases: either $x \leq \langle b_i^j, \{ b_i^j \} \rangle$ or $x \nleq \langle b_i^j, \{ b_i^j \} \rangle$. Suppose first that $x \leq \langle b_i^j, \{ b_i^j \} \rangle$. By Proposition \ref{Prop : about Fn}(\ref{item : basic structure of P(X) : 3}) we have $\langle b_i^j, \{ b_i^j \} \rangle \in \max \P(\mathsf{2}^k)$. Consequently, $d(\langle b_i^j, \{ b_i^j \} \rangle) = \{ w_i \}$ and, therefore, $e(\langle b_i^j, \{ b_i^j \} \rangle) = \{ w_i \}$. By the definition of $\hat{q}_j$ this implies $\hat{q}_j(\langle b_i^j, \{ b_i^j \} \rangle) = w_i$. Hence, we are done letting $y \coloneqq \langle b_i^j, \{ b_i^j \} \rangle$. Then we consider the case where  $x \nleq \langle b_i^j, \{ b_i^j \} \rangle$. Together with $x \in V_j$ and the definition of $V_j^+$, this implies $x \notin V_j^+$. Since $x \notin \max \P(\mathsf{2}^k)$, the definition of $\hat{q}_j$ guarantees that $\hat{q}_j(x) = s_{v_j, e(x)}$. As $\hat{q}_j(x) \leq w_i$, this yields $s_{v_j, e(x)} \leq w_i$. By condition (\ref{item : def : 1}) of Definition \ref{def} and the assumption that $w_i \in \max X$ we obtain $w_i \in e(x)$. On the other hand, from $x \nleq \langle b_i^j, \{ b_i^j \} \rangle$ it follows that $w_i \notin d(x)$. Thus, $w_i \in e(x) - d(x)$. By the definition of $e(x)$ this ensures that $i = 1$ and there exists $y \in \max {\uparrow}x$ such that $y \notin \{ \langle b_h^j, \{b_h^j\} \rangle : h \leq n+ 1 \}$. Since $y \in \max \P(\mathsf{2}^k)$, this implies $d(y) = \emptyset$ and $e(y) = \{ w_1 \}$. Hence,  from the definition of $\hat{q}_j$ it follows that  $\hat{q}_j(y) = w_1 = w_i$.
\end{proof}

By Claim \ref{Claim : disjoint} the following set is an upset of $\P(\mathsf{2}^k)$:
\[
U \coloneqq W_1 \cup \dots \cup W_{n+1} \cup V_1 \cup \dots \cup V_t.
\]
Moreover, from Claims \ref{Claim : disjoint} and \ref{Claim : the two maps} it follows that
\[
r \coloneqq \hat{p}_1 \cup \dots \cup \hat{p}_{n+1} \cup \hat{q}_1 \cup \dots \cup \hat{q}_t
\]
is a well-defined weak p-morphism from $U$ to $X$. In addition, $r \colon U \to X$ is surjective by equation (\ref{Eq : the partition of X}) and Claim \ref{Claim : the two maps}. Hence,  the map $r$ can be extended to a surjective weak p-morphism $r^+ \colon \P(\mathsf{2}^k) \to X$ by Proposition \ref{Prop : extension}.
\end{proof}

\section{Axiomatization and decidability}

From Theorem \ref{Thm : MAIN} it follows that the universal theory $\mathsf{Th}_\forall(\boldsymbol{F}_\pDL(\aleph_0))$ is decidable, as we proceed to explain.

\begin{Definition}
Let $\A \in \pDL$ be finite and with universe $A = \{ a_1, \dots, a_n \}$. Then
\benroman
\item the \emph{positive atomic diagram} of $\A$ is the set of equations 
\begin{align*}
\mathsf{diag}(\A)^+ \coloneqq \{ f(x_{i_1}, \dots, x_{i_m}) \thickapprox x_k : & \text{ }f \in \{ \land, \lor, \lnot, 0, 1 \} \text{ and the arity of $f$ is $m$},\\
& \text{ } i_1, \dots, i_m, k \leq n, \text{ and } f^{\A}(a_{i_1}, \dots, a_{i_m}) = a_k \};
\end{align*}
\item the \emph{negative atomic diagram} of $\A$ is the set of negated equations 
\[
\mathsf{diag}(\A)^- \coloneqq \{ x_m \not \thickapprox x_k : m < k \leq n\};
\]
\item the \emph{atomic diagram} of $\A$ is the set of formulas 
\[
\mathsf{diag}(\A) \coloneqq \mathsf{diag}(\A)^+ \cup \mathsf{diag}(\A)^-.
\]
\eroman
\end{Definition}
\noindent Notice that $\mathsf{diag}(\A)$ is a finite set of formulas with free variables among $x_1, \dots, x_n$. 

We will denote the first-order implication, conjunction, and disjunction by $\Rightarrow$, $\fowedge$, and $\fovee$, respectively. The following is an immediate consequence of the definition of an atomic diagram (see, e.g., \cite[Prop.~2.1.8]{CK90}).

\begin{Proposition}\label{Prop : embedding vs existential sentenes}
Let $\A, \B \in \pDL$ with $\A$ finite. Then $\A$ embeds into $\B$ if and only if
\[
\B \vDash \exists x_1, \dots, x_n \fobigwedge \mathsf{diag}(\A).
\]
\end{Proposition}

Recall that $\pDL$ is a finitely axiomatizable class, so let $\Sigma$ be a finite set of axioms for $\pDL$.

\begin{Proposition}\label{Prop : the theory is RAxiom}
The theory $\mathsf{Th}_\forall(\boldsymbol{F}_\pDL(\aleph_0))$ is recursively axiomatizable by
\[
\Sigma \cup \{ \lnot \exists x_1, \dots, x_n \fobigwedge \mathsf{diag}(\A) : \A \text{ belongs to $\pDL$, is finite, and $\A_*$ lacks a free skeleton}\}.
\]
\end{Proposition}

\begin{proof}
From Theorem \ref{Thm : MAIN} and Proposition \ref{Prop : embedding vs existential sentenes} it follows that the set of formulas in the statement axiomatizes $\mathsf{Th}_\forall(\boldsymbol{F}_\pDL(\aleph_0))$. Therefore, it only remains to prove that this set is recursive. This follows from the fact that $\Sigma$ is finite and that we can check mechanically whether a finite $\A \in \pDL$ has a free skeleton by inspecting the poset $\A_\ast$.
\end{proof}

Since $\pDL$ is a finitely axiomatizable and locally finite variety of finite type, we can apply Theorem \ref{Thm : recursively axiomatizable implies decidable} and Proposition  \ref{Prop : the theory is RAxiom}, obtaining the desired conclusion.

\begin{Theorem}\label{Thm : decidable}
The theory $\mathsf{Th}_\forall(\boldsymbol{F}_\pDL(\aleph_0))$ is decidable.
\end{Theorem}

While Proposition \ref{Prop : the theory is RAxiom} provides indeed a recursive axiomatization for $\mathsf{Th}_\forall(\boldsymbol{F}_\pDL(\aleph_0))$, this axiomatization has the obvious flaw of being a ``brute force'' one, in the sense that it consists in prohibiting one by one all the finite pseudocomplemented distributive lattices that are not models of the theory. In the rest of this section, we will amend this by presenting an alternative axiomatization of $\mathsf{Th}_\forall(\boldsymbol{F}_\pDL(\aleph_0))$ which, although still infinite, captures the idea of ``having a free skeleton'' in a more cogent way.

To this end, it is convenient to recall the definition of some terms introduced in \cite{KS24}.\ Let $\vec{x} = \langle x_1, \dots, x_n\rangle$ be a tuple of variables and denote the powerset of $\{1,\dots, n\}$ by $\wp(n)$. 
For every $T \in \wp(n)$ consider the term
\[
a_T(\vec{x}) \coloneqq \bigwedge_{i \in T} x_i \wedge \bigwedge_{i \notin T} \neg x_i.
\]
Then let 
\[
\mathcal{S}(n)\coloneqq \lbrace\langle L,\T \rangle :  L \in \wp(n), \ \T \subseteq \wp(n), \text{ and } \varnothing \ne L \subseteq \bigcap \T\rbrace.
\] 
For each $\langle L, \T \rangle \in \mathcal{S}(n)$ consider the term
\[
\plt(\vec{x}) \coloneqq \bigwedge_{i \in L} x_i \wedge \neg \neg \left( \bigvee_{T \in \T} a_T(\vec{x}) \right).
\]
Moreover, define a partial order $\leq$ on $\mathcal{S}(n)$ by setting
\[
\langle L,\T\rangle \leq \langle L',\T'\rangle \iff L \supseteq L'  \text{ and }\T \subseteq \T'.
\]

 We  will rely on the following observation, which stems from \cite{KS24}. Since its proof requires a detour from the issue under consideration, we decided to include it in the Appendix.

\begin{Proposition}\label{Prop : atoms jirr in fin gen}
Let $\A \in \pDL$ be generated by $a_1, \dots, a_n$ and let $\vec{a}=\langle a_1, \dots, a_n\rangle$. The following conditions hold for every $b \in A$:
\benroman
\item\label{item : atoms jirr in fin gen : 1} $b \in \At(\A)$ if and only if $b \ne 0$ and there exists $T \in \wp(n)$ such that $b=a_T(\vec{a})$;
\item\label{item : atoms jirr in fin gen : 2} $b \in \Jirr(\A)$ if and only if there exists $\langle L,T\rangle \in \mathcal{S}(n)$ such that 
\[
b=\plt(\vec{a}) \quad \text{and} \quad b \ne \bigvee \{p_{L'}^{\T'}(\vec{a}) : \langle L', \T' \rangle \in \mathcal{S}(n) \text{ and } \langle L',\T'\rangle < \langle L,\T\rangle \}.
\]
\eroman
\end{Proposition}

The description of $\At(\A)$ and $\Jirr(\A)$ in the above result can be expressed with first-order formulas in the language of pseudocomplemented distributive lattices as follows.

\begin{Definition}
Given $n \in \mathbb{Z}^+$ and a term $t=t(\vec{x})$ with $\vec{x} = \langle x_1, \dots, x_n\rangle$, let
\begin{align*}
\At_{t, n}(\vec{x}) & \coloneqq t \not\approx 0 \fowedge \fobigvee_{T \in \wp(n)} t \approx a_T(\vec{x});\\
\Jirr_{t, n}(\vec{x}) & \coloneqq \fobigvee_{\langle L,\T\rangle \in \mathcal{S}(n)} \left(t \approx \plt(\vec{x}) \fowedge t \not\approx \bigvee \{p_{L'}^{\T'}(\vec{x}) : \langle L',T'\rangle \in \mathcal{S}(n) \text{ and }\langle L',T'\rangle < \langle L,T\rangle \}\right).
\end{align*}
\end{Definition}

The following is an immediate consequence of Proposition~\ref{Prop : atoms jirr in fin gen}.

\begin{Corollary}\label{Cor : Att and Jirrt}
Let $\A \in \pDL$ be generated by $a_1, \dots, a_n$ and let $\vec{a}=\langle a_1, \dots, a_n\rangle$. The following conditions hold for every term $t(x_1, \dots, x_n)$:
\benroman
\item $t(\vec{a}) \in \At(\A)$ if and only if $\A \vDash \At_{t, n}(\vec{a})$;
\item $t(\vec{a}) \in \Jirr(\A)$ if and only if $\A \vDash \Jirr_{t, n}(\vec{a})$.
\eroman
\end{Corollary}

When $\vec{x}$ and $\vec{y}$ are disjoint tuples of variables, $t$ a term, and $G$ a first-order formula, we will often write $t(\vec{x},\vec{y})$ and $G(\vec{x},\vec{y})$ to denote $t(\vec{u})$ and $G(\vec{u})$, where $\vec{u}$ is the concatenation of $\vec{x}$ and $\vec{y}$.

\begin{Definition}\label{def:F(x,y)}
Let $P(\vec{x},\vec{y})$ and $Q(\vec{x},\vec{y},\vec{z})$ be quantifier-free formulas, where $\vec{x}= \langle x_1, \dots, x_n\rangle$, $\vec{y} = \langle y_1, \dots, y_m\rangle$, and $\vec{z} = \langle z_1, \dots, z_k \rangle$ are disjoint tuples of variables. Then let $F_{n, m, k}(\vec{x},\vec{y})$ be the quantifier-free formula $A \Rightarrow B$ with
\begin{align*}
A &\coloneqq \Jirr_{x_1, n+m}(\vec{x}, \vec{y})  \fowedge \fobigwedge_{i=1}^m \At_{y_i, n+m}(\vec{x}, \vec{y}) \fowedge P(\vec{x}, \vec{y});\\
B &\coloneqq \fobigvee_{\mathcal{S}' \in \mathcal{S}({n+m})^k} \Bigg( Q(\vec{x},\vec{y}, p_1, \dots, p_k) \fowedge  \fobigwedge_{j=1}^k \Jirr_{p_j, n+m}(\vec{x}, \vec{y}) \Bigg),
\end{align*}
where $\mathcal{S}' = \langle \langle L_1,\T_1\rangle, \dots, \langle L_k,\T_k\rangle \rangle$ and $p_j$ denotes the term $p_{L_j}^{\T_j}(\vec{x},\vec{y})$ for every $j \leq k$.
\end{Definition}

\begin{Proposition}\label{Prop : correspondence qf formula}
Let $\A \in \pDL$, $\vec{a}=\langle a_1, \dots, a_n\rangle \in A^n$, and $\vec{b}=\langle b_1, \dots, b_m \rangle \in A^m$.\ Moreover, let $\B$ be the subalgebra of $\A$ generated by $a_1, \dots, a_n, b_1, \dots, b_m$. If $F_{n, m, k}(\vec{x},\vec{y})$ is as in Definition~\ref{def:F(x,y)}, then
\begin{align*}
\A \vDash F_{n, m, k}(\vec{a}, \vec{b}) \iff &\text{if }a_1 \in \Jirr(\B), \ b_1, \dots, b_m \in \At(\B), \text{ and } \B \vDash P(\vec{a}, \vec{b}),\\
& \text{then there exists } \vec{c} \in \Jirr(\B)^k \text{ such that } \B \vDash Q(\vec{a}, \vec{b}, \vec{c}).
\end{align*}
\end{Proposition}

\begin{proof}
We begin by observing that $\A \vDash F_{n, m, k}(\vec{a},\vec{b})$ if and only if $\B \vDash F_{n, m, k}(\vec{a},\vec{b})$ because $F_{n, m, k}(\vec{x},\vec{y})$ is quantifier-free and $a_1, \dots, a_n, b_1, \dots, b_m \in B$. Therefore, it will be enough to prove that
\begin{equation}\label{Eq : the splitted equation}
\begin{split}
\B \vDash F_{n, m, k}(\vec{a}, \vec{b}) \iff &\text{if }a_1 \in \Jirr(\B), \ b_1, \dots, b_m \in \At(\B), \text{ and } \B \vDash P(\vec{a}, \vec{b}),\\
& \text{then there exists } \vec{c} \in \Jirr(\B)^k \text{ such that } \B \vDash Q(\vec{a}, \vec{b}, \vec{c}).
\end{split}
\end{equation}
Since $\B$ is generated by $a_1, \dots, a_n, b_1, \dots, b_m$, the following conditions hold by Corollary~\ref{Cor : Att and Jirrt}:
\benroman
\item the antecedent of $F_{n, m, k}(\vec{a},\vec{b})$ holds in $\B$ if and only if $\B \vDash P(\vec{a},\vec{b})$ with $a_1 \in \Jirr(\B)$ and $b_1, \dots, b_m \in \At(\B)$;
\item the consequent of $F_{n, m, k}(\vec{a},\vec{b})$ holds in $\B$ if and only if there exist $\langle L_1, \T_1 \rangle, \dots, \langle L_k, \T_k\rangle \in \mathcal{S}(n+m)$ such that
\[
\B \vDash Q(\vec{a},\vec{b},p_{L_1}^{\T_1}(\vec{a},\vec{b}), \dots,p_{L_k}^{\T_k}(\vec{a},\vec{b})) \quad \text{with} \quad p_{L_1}^{\T_1}(\vec{a},\vec{b}), \dots,p_{L_k}^{\T_k}(\vec{a},\vec{b}) \in \Jirr(\B).
\]
\eroman
Therefore, in order to prove (\ref{Eq : the splitted equation}), it only remains to show that the following conditions are equivalent:
\benroman
\item[A.] there exists $\vec{c} \in \Jirr(\B)^k$ such that $\B \vDash Q(\vec{a}, \vec{b}, \vec{c})$;
\item[B.] there exist $\langle L_1, \T_1 \rangle, \dots, \langle L_k, \T_k\rangle \in \mathcal{S}(n+m)$ such that $\B \vDash Q(\vec{a},\vec{b},p_{L_1}^{\T_1}(\vec{a},\vec{b}), \dots,p_{L_k}^{\T_k}(\vec{a},\vec{b}))$ with $p_{L_1}^{\T_1}(\vec{a},\vec{b}), \dots,p_{L_k}^{\T_k}(\vec{a},\vec{b}) \in \Jirr(\B)$.
\eroman
Clearly, (B) implies (A). Then we prove the converse. Suppose that (A) holds, that is, there exists $\vec{c}=\langle c_1,\dots, c_k \rangle \in \Jirr(\B)^k$ such that $\B\vDash Q(\vec{a}, \vec{b}, \vec{c})$. Since $\B$ is generated by $a_1, \dots, a_n, b_1, \dots, b_m$,
 Proposition~\ref{Prop : atoms jirr in fin gen}(\ref{item : atoms jirr in fin gen : 2}) implies that for every $i \leq k$ there exists $\langle L_i, \T_i \rangle \in \mathcal{S}(n+m)$ such that $c_i=p_{L_i}^{\T_i}(\vec{a},\vec{b})$. Therefore, (B) holds as desired.
\end{proof}

Let $\vec{x}= \langle x_1, \dots, x_n\rangle$ and $\vec{y} = \langle y_1, \dots, y_m\rangle$ be disjoint tuples of variables. We will define two kinds of formulas:
\benroman 
\item for each $h \leq m$ let
\[
P_{n, m, h}(\vec{x},\vec{y}) \coloneqq \fobigwedge_{i=1}^h y_i \leq x_1 \fowedge \fobigwedge_{j=h+1}^m y_j \nleq x_1;
\]
\item for each $h \geq 1$ fix a bijection $f \colon \{1, \dots, 2^h-1\} \to \wp(h) - \{\varnothing\}$. Then for each tuple  of variables $\vec{z} = \langle z_1, \dots, z_{2^h-1} \rangle$ disjoint from $\vec{x}$ and $\vec{y}$ let $Q_{n, m, h}(\vec{x},\vec{y},\vec{z})$ be the formula
\[
 \bigg(\fobigwedge_{i \leq 2^h-1} z_i \leq x_1\bigg) 
\fowedge \bigg(\fobigwedge_{\substack{i,j \leq 2^h-1\\ f(i) \subseteq f(j)}} z_i \leq z_j\bigg)
\fowedge \bigg(\fobigwedge_{\substack{i \leq 2^h-1, \ j \leq m\\j \in f(i)}} y_j \leq z_i\bigg) 
\fowedge \bigg(\fobigwedge_{\substack{i \leq 2^h-1, \ j \leq m\\j \notin f(i)}} y_j \nleq z_i\bigg).
\]
\eroman

\begin{Definition}
Let $\vec{x}= \langle x_1, \dots, x_n\rangle$ and $\vec{y} = \langle y_1, \dots, y_m\rangle$ be disjoint tuples of variables. Then for each $1 \leq h \leq m$ let $\textup{FS}_{n, m, 2^h-1}(\vec{x},\vec{y})$ be the quantifier-free formula obtained from $P_{n ,m, h}(\vec{x},\vec{y})$ and $Q_{n, m, h}(\vec{x},\vec{y}, \vec{z})$ as in Definition~\ref{def:F(x,y)}, where $\vec{z}= \langle z_1, \dots, z_{2^h-1}\rangle$ is a tuple of variables disjoint from $\vec{x}$ and $\vec{y}$.
\end{Definition}

We recall that the \emph{universal closure} of a quantifier-free formula $P(\vec{x})$ is the universal sentence $\forall \vec{x} P(\vec{x})$. We are now ready to present the new set of axioms for $\mathsf{Th}_\forall(\boldsymbol{F}_\pDL(\aleph_0))$.

\begin{Definition}
Consider the quantifier-free formula
\[
\textup{DN}(x,y,z) \coloneqq \neg \neg x = y \vee z \Rightarrow (y = \neg \neg x \sqcup z = \neg \neg x).
\]
Then let $\Sigma$ be the set of universal closures of the formulas in
\[
\{ 0 \not \thickapprox 1,  \textup{DN}(x,y,z) \} \cup \{ \textup{FS}_{n, m, 2^h-1}(\vec{x},\vec{y}) : n, m, h \in \mathbb{Z}^+ \text{ and } h \leq m \}.
\]
\end{Definition}

\begin{Proposition}\label{Prop : FSh, DN and free skeleton}
Let $\A \in \pDL$. Then $\A \vDash \Sigma$ if and only if $\B_*$ has a free skeleton for every finite subalgebra $\B$ of $\A$.
\end{Proposition}

\begin{proof}
It suffices to show that $\A \vDash \Sigma$ if and only if every finite subalgebra $\B$ of $\A$ is nontrivial and satisfies conditions (\ref{item : equiv free skel : 1}), (\ref{item : equiv free skel : 2}), and (\ref{item : equiv free skel : 3}) of Proposition~\ref{Prop : equiv free skel}.
To prove the implication from left to right, assume that $\A \vDash \Sigma$ and consider a finite subalgebra $\B$ of $\A$. As $\A \vDash 0 \not \thickapprox 1$, the algebra $\B$ is nontrivial. Then let $a \in \Jirr(\B)$ and nonempty $Y, Z \subseteq \At(\B) \cap \down a$. As $a \in \Jirr(\B)$, we have $a \ne 0$. Since $\B$ is finite, this yields $\At(\B) \cap \down a \ne \emptyset$. Fix an enumeration $\{ b_1, \dots, b_m\}$ of $\At(\B)$ such that $\{b_1, \dots, b_h\} = \At(\B) \cap \down a$. Then consider $a_2, \dots, a_n \in B$ such that $a,a_2, \dots, a_n,b_1, \dots, b_m$ generate $\B$. Let $\vec{a} \coloneqq \langle a, a_2, \dots, a_n \rangle$ and $\vec{b} \coloneqq \langle b_1, \dots, b_m \rangle$. Since $\A \vDash \Sigma$, we have that $\A \vDash \textup{FS}_{n, m, 2^h-1}(\vec{a},\vec{b})$.
Furthermore, from $\At(\B) = \{ b_1, \dots, b_m\}$ and  $\At(\B) \cap \down a = \{b_1, \dots, b_h\}$ it follows that
\[
b_1, \dots, b_h \leq a \quad \text{and} \quad b_{h+1}, \dots, b_m \nleq a.
\]
By the definition of $P_{n, m, h}$ this amounts to $\B \vDash P_{n, m, h}(\vec{a}, \vec{b})$. Therefore, Proposition~\ref{Prop : correspondence qf formula}  implies that there exist $c_1, \dots, c_{2^h-1} \in \Jirr(\B)$ such that $\B \vDash Q_{n, m, h}(\vec{a},\vec{b}, c_1, \dots, c_{2^h-1})$. By the definition of $Q_{n, m, h}$ this means that there exists a bijection $f \colon \{1, \dots, 2^h-1\} \to \wp(h) - \{\varnothing\}$ such that
\benormal
\item\label{item : conditions FSh : 1} $c_i \leq a$ for every $i \leq 2^h-1$;
\item\label{item : conditions FSh : 2} $c_i \leq c_j$ for all $i, j \leq 2^h-1$ such that $f(i) \subseteq f(j)$;
\item\label{item : conditions FSh : 3} $b_j \leq c_i$ for all $i \leq 2^h-1$ and $j \leq m$ such that $j \in f(i)$;
\item\label{item : conditions FSh : 4} $b_j \nleq c_i$ for all $i \leq 2^h-1$ and $j \leq m$ such that $j \notin f(i)$.
\enormal
For each $i \leq 2^h-1$ let $c_{a,f(i)} \coloneqq c_i$. Since $f \colon \{1, \dots, 2^h-1\} \to \wp(h) - \{\varnothing\}$ is a bijection and $\At(\B) \cap \down a = \{b_1, \dots, b_h\}$, this is a definition of $c_{a, Y}$ for each nonempty $Y \subseteq \At(\B) \cap \down a$.

From (\ref{item : conditions FSh : 1}), (\ref{item : conditions FSh : 3}), and (\ref{item : conditions FSh : 4}) it follows that for every nonempty $Y \subseteq \At(\B) \cap \down a$ we have that $c_{a,Y} \leq a$ and $\At(\B) \cap \down c_{a,Y} = Y$. Moreover,  (\ref{item : conditions FSh : 2}) yields $c_{a,Y} \leq c_{a,Z}$ for each nonempty $Y,Z \subseteq \At(\B) \cap \down a$ such that $Y \subseteq Z$. Thus, $\B$ satisfies conditions (\ref{item : equiv free skel : 1}) and (\ref{item : equiv free skel : 2}) of Proposition~\ref{Prop : equiv free skel}. Lastly, since  $\A$ validates $\forall x,y,z \,\textup{DN}(x,y,z)$ by hypothesis, $\neg \neg b \in \Jirr(\A)$ for every $b \in A - \{ 0 \}$. Consequently, $\neg \neg b \in \Jirr(\B)$ for every $b \in B - \{ 0 \}$. Thus, $\B$ satisfies condition (\ref{item : equiv free skel : 3}) of Proposition~\ref{Prop : equiv free skel} as well. Hence, we conclude that $\B_*$ has a free skeleton.

To prove the other implication, suppose that every finite subalgebra $\B$ of $\A$ is nontrivial and satisfies conditions (\ref{item : equiv free skel : 1}), (\ref{item : equiv free skel : 2}), and (\ref{item : equiv free skel : 3}) of Proposition~\ref{Prop : equiv free skel}. First, as every finite subalgebra of $\A$ is nontrivial, so is the two-element subalgebra with universe $\{ 0, 1 \}$. It follows that $\A \vDash 0 \not \thickapprox 1$. Then we prove that $\A \vDash \forall x,y,z \,\textup{DN}(x,y,z)$. It suffices to show that $\neg \neg b \in \Jirr(\A)$ for every $b \in A - \{ 0 \}$. Consider $b, c,d \in A$ such that $b \ne 0$ and $\neg \neg b = c \vee d$ and let $\B$ be the subalgebra of $\A$ generated by the elements $b$, $c$, and $d$. Then $\B$ is finite because $\pDL$ is locally finite. Thus, condition (\ref{item : equiv free skel : 3}) of Proposition~\ref{Prop : equiv free skel} holds in $\B$, and hence $\neg \neg b \in \Jirr(\B)$. It follows that $\neg \neg b = c$ or $\neg \neg b = d$. As $b \ne 0$ by assumption, this shows that $\neg \neg b \in \Jirr(\A)$.

It only remains to prove that $\A$ validates $\forall \vec{x}, \vec{y} \, \textup{FS}_{n, m, 2^h-1}(\vec{x}, \vec{y})$ for each $n, m, h \in \mathbb{Z}^+$ such that $h \leq m$. To this end, let $\vec{a} = \langle a_1, \dots, a_n \rangle \in A^n$, $\vec{b} = \langle b_1, \dots, b_m \rangle \in A^m$, and $1 \leq h \leq m$. We will show that $\A \vDash \textup{FS}_{n, m, 2^h-1}(\vec{a},\vec{b})$ using Proposition~\ref{Prop : correspondence qf formula}. Let $\B$ be the subalgebra of $\A$ generated by $a_1, \dots, a_n, b_1, \dots, b_m$. 
Then $\B$ is finite because $\pDL$ is locally finite. Suppose that $a_1 \in \Jirr(\B)$, $b_1, \dots, b_m \in \At(\B)$, and $\B \vDash P_{n, m, h}(\vec{a},\vec{b})$. We only need to prove that there exists $\vec{c}\in \Jirr(\B)^{2^h-1}$ such that $\B \vDash Q_{n, m, h}(\vec{a}, \vec{b}, \vec{c})$. First, observe that $b_i \leq a_1$ if and only if $i \leq h$ because $\B \vDash P_{n, m, h}(\vec{a}, \vec{b})$. Consequently, $\{ b_1, \dots, b_h \} \subseteq \At(\B) \cap \down a_1$.  Thus, if $Y \subseteq \{ b_1, \dots, b_h\}$, then $Y \subseteq \At(\B) \cap \down a_1$. In view of condition (\ref{item : equiv free skel : 1}) of Proposition~\ref{Prop : equiv free skel} and $\{ b_1, \dots, b_h \} \subseteq \At(\B) \cap \down a_1$, for every nonempty $Y \subseteq \{ b_1, \dots, b_h\}$ there exists $c_{a_1,Y} \in \Jirr(\B)$ such that $c_{a_1,Y} \leq a_1$ and $\At(\B) \cap \down c_{a_1,Y}=Y$. In particular, $b_i \leq c_{a_1,Y}$ if and only if
 $b_i \in Y$. Similarly, condition (\ref{item : equiv free skel : 2}) of Proposition~\ref{Prop : equiv free skel} yields that for all nonempty $Y,Z \subseteq \{ b_1, \dots, b_h\}$ with $Y \subseteq Z$ we have $c_{a_1,Y} \leq c_{a_1,Z}$. Now, let $f \colon \{1, \dots, 2^h-1\} \to \wp(h) - \{\varnothing\}$ 
 be the bijection underlying the definition of $Q_{n, m, h}$. Then define $c_i\coloneqq c_{a_1,Y}$ with $Y \coloneqq \{b_j : j \in f(i)\}$ for each $i \leq 2^h-1$. From the definition of $Q_{n, m, h}$ it follows that $\B \vDash Q_h(\vec{a}, \vec{b},c_1,\dots,c_{2^h-1})$ as desired.
\end{proof}

As a consequence, we obtain the desired result.

\begin{Theorem}\label{Thm : improved axiomatization}
$\mathsf{Th}_\forall(\boldsymbol{F}_\mathsf{V}(\aleph_0))$ is axiomatized by $\Sigma$.
\end{Theorem}
\begin{proof}
Let $\A \in \pDL$.
By Theorem~\ref{Thm : MAIN} we have that $\A \vDash \mathsf{Th}_\forall(\boldsymbol{F}_\mathsf{V}(\aleph_0))$ if and only if $\B_*$ has a free skeleton for every finite subalgebra $\B$ of $\A$. Proposition~\ref{Prop : FSh, DN and free skeleton} yields that the latter condition holds if and only if $\A \vDash \Sigma$. Hence, $\mathsf{Th}_\forall(\boldsymbol{F}_\mathsf{V}(\aleph_0))$ is axiomatized by $\Sigma$.
\end{proof}

\begin{Remark}\label{Rem : the final remark}
We recall that a universal sentence is said to be \emph{admissible} in a variety $\mathsf{V}$ when it holds in the free algebra $\boldsymbol{F}_\mathsf{V}(\aleph_0)$ (see, e.g., \cite[Thm.~2]{CM15}). Moreover, an axiomatization of $\mathsf{Th}_\forall(\boldsymbol{F}_\mathsf{V}(\aleph_0))$ is often called a \emph{basis} for the universal sentences admissible in $\mathsf{V}$. In this parlance, Theorem \ref{Thm : decidable} states that the problem of determining whether a universal sentence is admissible in $\pDL$ is decidable and Theorem \ref{Thm : improved axiomatization} identifies a basis for the universal sentences admissible in $\pDL$.
These results, however, do not admit an immediate translation
in terms of the admissibility of multiconclusion rules of the implication-free fragment of the
intuitionistic propositional calculus \textsf{IPC} because this fragment is not algebraizable in the sense
of \cite{BP89} (for more details, see \cite{Raf16}).
 Nonetheless, we remark that the single-conclusion implication-free fragment of \textsf{IPC} is structurally complete \cite{Min72}.
\qed
\end{Remark}

We conclude by outlining some open problems and directions for future investigation.

\benroman
\item It remains open whether $\mathsf{Th}_\forall(\boldsymbol{F}_\mathsf{V}(\aleph_0))$ admits a finite axiomatization. We expect the answer to be negative. Nevertheless, it would be interesting to explore whether the axiomatization can be refined into an independent one, possibly following an approach similar to the one of \cite{Jer08}.
\item We have established that $\mathsf{Th}_\forall(\boldsymbol{F}_\mathsf{V}(\aleph_0))$ is decidable. Its complexity, however, remains an open problem. In contrast, the complexity of the universal theory of free Heyting algebras is known to be co-NEXP-complete \cite{Jer07} and PSPACE-complete for the free algebras of the implication-negation fragment of \textsf{IPC} \cite{CM10}. 
\item It is shown in \cite{Lee70} that the subvariety lattice of $\pDL$ forms a chain of order type $\omega+1$, where all the proper subvarieties are finitely generated. The dual description of the free algebras in $\pDL$ can be easily adapted to each subvariety (see \cite{DG80}). However, the problem of axiomatizing the universal theory of the free algebras in subvarieties of $\pDL$ remains open, as our approach may require some adaptation.
\eroman

\paragraph{\bfseries Acknowledgements.}
The first author was supported by the ``National
Group for Algebraic and Geometric Structures, and their Applications'' (GNSAGA - INdAM). The second author was supported by the proyecto PID$2022$-$141529$NB-C$21$ de investigación financiado por
MICIU/AEI/$10$.$13039$/$501100011033$ y por FEDER, UE. He was also supported by the Research
Group in Mathematical Logic, $2021$SGR$00348$ funded by the Agency for Management of University and Research Grants of the Government of Catalonia, and by the MSCA-RISE-Marie Skłodowska-Curie Research and Innovation Staff Exchange (RISE)
project MOSAIC $101007627$ funded by Horizon $2020$ of the European Union.

\appendix

\section*{Appendix}
\renewcommand{\theTheorem}{A.\arabic{Theorem}}
\setcounter{Theorem}{0}

Our aim is to prove Proposition \ref{Prop : atoms jirr in fin gen}. To this end, we will rely on the next technical observation.

\begin{Lemma}\label{Lem : at jirr onto maps}
Let $\A,\B \in \pDL$ be finite, $h \colon \A \to \B$ a surjective homomorphism, and $b \in B$. Then the following conditions hold:
\benroman
\item\label{item : at jirr onto maps : 1} $b \in \At(\B)$ if and only if $b \ne 0$ and there exists $a \in \At(\A)$ such that $h(a)=b$;
\item\label{item : at jirr onto maps : 2} $b \in \Jirr(\B)$ if and only if there exists $a \in \Jirr(\A)$ such that 
\[
h(a) =b \quad \text{and} \quad b \ne \bigvee \{ h(c) : c \in \Jirr(\A) \text{ and } c < a \}.
\]
\eroman
\end{Lemma}
\begin{proof}
We first establish the following claim.

\begin{Claim}\label{Claim : at jirr onto maps}
Let $a = \bigwedge h^{-1}[\up b]$. Then $h(a)=b$ and for every $c < a$ we have $h(c) < b$. In addition, if $b \in \Jirr(\B)$, then $a \in \Jirr(\A)$.
\end{Claim}

\begin{proof}[Proof of the Claim]
Observe that $\bigwedge h^{-1}[\up b]$ exists because $\A$ is finite. Since $h$ is a homomorphism of bounded distributive lattices, $h^{-1}[\up b]$ is a filter of $\A$. Then $a =\bigwedge h^{-1}[\up b] \in h^{-1}[\up b]$, and so $b \leq h(a)$. As $h$ is onto, there exists $c \in A$ such that $h(c)=b$. Then $c \in h^{-1}[\up b]$. Thus, $a =  \bigwedge h^{-1}[\up b] \leq c$, and hence $b \leq h(a) \leq h(c)=b$. It follows that $h(a)=b$.

Then consider $c \in A$ such that $c < a$. First, observe that $h(c) \leq h(a) = b$. We will prove that $h(c) < b$. Suppose the contrary. Together with $h(c) \leq b$, this yields $h(c) = b$. By the definition of $a$ we obtain $a \leq c$, a contradiction with the assumption that $c < a$. Hence, we conclude that $h(c) < b$. 

Lastly, suppose that $b \in \Jirr(\B)$. We will show that $a \in \Jirr(\A)$. Since $b$ is join-irreducible, $\up b$ is a prime filter of $\B$ by Proposition~\ref{Prop : correspondence prime filters join irr}. Then $h^{-1}[\up b]$ is a prime filter of $\A$ because $h$ is a homomorphism of bounded distributive lattices. Thus, $a = \bigwedge h^{-1}[\up b]$ is also join-irreducible by Proposition~\ref{Prop : correspondence prime filters join irr}.
\end{proof}

(\ref{item : at jirr onto maps : 1}): We first prove the implication from left to right. Let $b \in \At(\B)$. Then $b \ne 0$. Define $a \coloneqq \bigwedge h^{-1}[\up b]$. By the Claim we have $h(a)=b$. Therefore, it only remains to show that $a \in \At(\A)$. First, observe that $a \ne 0$ because $h(a)=b$ and $b \ne 0$. Then consider $c \in A$ such that $c < a$. By the Claim we have $h(c) < b$. Since $b$ is an atom by assumption, it follows that $h(c)=0$. Then $h(\neg c)=\neg h(c)= \neg 0 = 1 \geq b$.  By the definition of $a$ this yields $a \leq \neg c$ and, therefore, $a \wedge c = 0$. Together with $c < a$, this implies $c=a \wedge c =0$. Hence, we conclude that $a \in \At(\A)$ as desired. 

To prove the converse implication, consider $a \in \At(\A)$ such that $h(a) \ne 0$. We need to show that $h(a) \in \At(\B)$. Since $h(a) \ne 0$, it suffices to show that for every $d \in B$ such that $d \leq h(a)$ either $d= h(a)$ or $d = 0$ holds true. Consider $d \in B$ such that $d \leq h(a)$. Since $h$ is surjective, there exists $c \in A$ such that $h(c)=d$. Then $h(c) \leq h(a)$, which implies $h(c)=h(c) \wedge h(a)=h(c \wedge a)$. As $a$ is an atom, we have that either $c \wedge a = a$ or $c \wedge a = 0$. Together with $d=h(c)=h(c \wedge a)$, this yields that either $d=h(a)$ or $d=h(0)=0$. Hence, $h(a) \in \At(\B)$.

(\ref{item : at jirr onto maps : 2}): We begin by proving the implication from left to right. Let $b \in \Jirr(\B)$. Define $a \coloneqq  \bigwedge h^{-1}[\up b]$. By the Claim  we have $h(a)=b$ and $a \in \Jirr(\A)$. It only remains to show that $b \ne \bigvee \{ h(c) : c \in \Jirr(\A) \text{ and } c < a \}$. To this end, consider $c \in \Jirr(\A)$ such that $c < a$. By the Claim we have $h(c) < b$. Since $b$ is join-irreducible, it follows that $b \ne \bigvee \{ h(c) : c \in \Jirr(\A) \text{ and } c < a \}$.

Then we prove the converse implication. Consider an element 
$a \in \Jirr(\A)$ for which $h(a) \ne \bigvee \{ h(c) : c \in \Jirr(\A) \text{ and } c < a \}$. In order to establish that $h(a) \in \Jirr(\B)$, we first prove that 
\[
\{d : d \in \Jirr(\B) \text{ and } d < h(a) \} \subseteq \{ h(c) : c \in \Jirr(\A) \text{ and } c < a \}.
\]
Let $d \in \Jirr(\B)$ and assume that $d < h(a)$. Then consider $c \coloneqq \bigwedge h^{-1}[\up d]$. By the Claim we get $h(c)=d$ and $c \in \Jirr(\A)$.  Since $d < h(a)$, we have $a \in h^{-1}[\up d]$ and, therefore, $c  = \bigwedge h^{-1}[\up d]\leq a$. From $h(c)=d < h(a)$ it follows that $c \ne a$. Together with $c \leq a$, this yields $c < a$. Thus, $d=h(c)$ with $c \in \Jirr(\A)$ and $c < a$. This establishes the above display, which yields 
\begin{equation}\label{Eq : the second to last step in this proof}
\bigvee \{d : d \in \Jirr(\B) \text{ and } d < h(a) \} \leq \bigvee \{ h(c) : c \in \Jirr(\A) \text{ and } c < a \} < h(a),
\end{equation}
where the last inequality holds because  we assumed $h(a) \ne \bigvee \{ h(c) : c \in \Jirr(\A) \text{ and } c < a \}$ and $h$ is order preserving.
As $\B$ is finite, each of its elements is a join of join-irreducibles. Consequently,
\[
\bigvee \{e : e \in B \text{ and } e < h(a) \} = \bigvee \{d : d \in \Jirr(\B) \text{ and } d < h(a) \}.
\]
Together with (\ref{Eq : the second to last step in this proof}), the above display yields $\bigvee \{e : e \in B \text{ and } e < h(a) \} < h(a)$. Thus, we conclude that $h(a) \in \Jirr(\B)$.
\end{proof}

Now, recall that the free generators of $\boldsymbol{F}_\pDL(n)$ are $x_0/\theta_n, \dots, x_{n-1}/\theta_n$ and let
\[
\vec{x}/\theta_n \coloneqq \langle x_0/\theta_n, \dots, x_{n-1}/\theta_n\rangle.
\]
The atoms and join-irreducibles of $\boldsymbol{F}_\pDL(n)$ can be described as follows.

\begin{Theorem}\label{Thm : atoms and jirr in free}
The following conditions hold for every $n \in \mathbb{Z}^+$:
\benroman
\item\label{item : atoms and jirr in free : 1} $\At(\boldsymbol{F}_\pDL(n))=\{a_T(\vec{x}/\theta_n) : T \in \wp(n)\}$;
\item\label{item : atoms and jirr in free : 2} $\Jirr(\boldsymbol{F}_\pDL(n))=\{\plt(\vec{x}/\theta_n) : \langle L,\T\rangle \in \mathcal{S}(n) \}$; 
\item\label{item : atoms and jirr in free : 3} for all $\langle L, \T \rangle, \langle L', \T' \rangle \in \mathcal{S}(n)$ we have
\[
\plt(\vec{x}/\theta_n) \leq p_{L'}^{\T'}(\vec{x}/\theta_n) \quad \text{if and only if} \quad\langle L,\T\rangle \leq \langle L',\T'\rangle.
\]
\eroman
\end{Theorem}

\begin{proof}
For (\ref{item : atoms and jirr in free : 1}) and (\ref{item : atoms and jirr in free : 2}), see \cite[Cors.~5.1 and~5.5(2)]{KS24}.

(\ref{item : atoms and jirr in free : 3}): In the proof of \cite[Thm.~5.7(2)]{KS24}, it is shown that $\plt(\vec{x}/\theta_n) \leq p_{L'}^{\T'}(\vec{x}/\theta_n)$ if and only if $\mu_\T^L \leq^{\textbf{Cm}} \mu_{\T'}^{L'}$ (for the definition of the latter, see \cite{KS24}).
Therefore, condition (\ref{item : atoms and jirr in free : 3}) is an immediate consequence of \cite[Thm.~5.7(2)]{KS24}.
\end{proof}

We are now ready to prove Proposition \ref{Prop : atoms jirr in fin gen}.

\begin{proof}
First, let $h \colon \boldsymbol{F}_\pDL(n) \to \A$ be the unique homomorphism such that $h(x_i/\theta_n)=a_{i+1}$ for each $i = 0, \dots, n-1$.
Recall from Theorem \ref{Thm : PDL locally finite} that $\pDL$ is locally finite. Therefore, the free algebra $\boldsymbol{F}_\pDL(n)$ is finite. Moreover, $h$ is surjective because $\A$ is generated by $a_1, \dots, a_n$. Hence, $h$ satisfies the hypotheses of Lemma~\ref{Lem : at jirr onto maps}. This fact will be used repeatedly in the rest of the proof.
 
(\ref{item : atoms jirr in fin gen : 1}): By Lemma~\ref{Lem : at jirr onto maps}(\ref{item : at jirr onto maps : 1}) we have that $b \in \At(\A)$ if and only if $b \ne 0$ and there exists $c \in \At(\boldsymbol{F}_\pDL(n))$ such that $b=h(c)$. By Theorem~\ref{Thm : atoms and jirr in free}(\ref{item : atoms and jirr in free : 1}) the atoms of $\boldsymbol{F}_\pDL(n)$ are exactly the elements of the form $a_T(\vec{x}/\theta_n)$ for some $T \in \wp(n)$. Moreover, the definition of $h$ implies that $h(a_T(\vec{x}/\theta_n)) = a_T(\vec{a})$ for every $T \in \wp(n)$. Hence, $b \in \At(\A)$ if and only if $b \ne 0$ and there exists $T \in \wp(n)$ such that $b=a_T(\vec{a})$.

(\ref{item : atoms jirr in fin gen : 2}): By Lemma~\ref{Lem : at jirr onto maps}(\ref{item : at jirr onto maps : 2}) we have that $b \in \Jirr(\A)$ if and only if there exists $d \in \Jirr(\boldsymbol{F}_\pDL(n))$ such that 
\[
h(d) = b \quad \text{and} \quad b \ne \bigvee \{ h(c) : c \in \Jirr(\boldsymbol{F}_\pDL(n)) \text{ and } c < d \}.
\]
By Theorem~\ref{Thm : atoms and jirr in free}(\ref{item : atoms and jirr in free : 2})  the join-irreducibles of $\boldsymbol{F}_\pDL(n)$ are exactly the elements of the form $\plt(\vec{x}/\theta_n)$ for some $\langle L, \T \rangle \in \mathcal{S}(n)$. As above, we have that $h(\plt(\vec{x}/\theta_n)) = \plt(\vec{a})$ for every $\langle L, \T \rangle \in \mathcal{S}(n)$. Lastly, recall from Theorem~\ref{Thm : atoms and jirr in free}(\ref{item : atoms and jirr in free : 3}) that $p_{L'}^{\T'}(\vec{x}/\theta_n) \leq \plt(\vec{x}/\theta_n)$ if and only if $\langle L',\T'\rangle \leq \langle L,\T\rangle$. Hence, $b \in \Jirr(\A)$ if and only if there exists $\langle L,T\rangle \in \mathcal{S}(n)$ such that 
\[
b=\plt(\vec{a}) \quad \text{and} \quad b \ne \bigvee \{p_{L'}^{\T'}(\vec{a}) : \langle L',\T'\rangle \in \mathcal{S}(n) \text{ and }\langle L',\T'\rangle < \langle L,\T\rangle \}.\qedhere
\]
\end{proof}

\bibliographystyle{plain}

\end{document}